\documentclass[a4paper,11pt,reqno]{amsart}
\usepackage{amssymb}
\usepackage{amsmath}
\usepackage{enumitem}
\usepackage{mathrsfs}
\usepackage[all]{xy}
\usepackage{mathtools}
\usepackage{hyperref}

\usepackage[OT2,T1]{fontenc}
\usepackage[british]{babel}

\usepackage[margin=1.3in]{geometry}

\numberwithin{equation}{section} \numberwithin{figure}{section}

\usepackage[usenames,dvipsnames]{color}

\DeclareMathOperator{\Pic}{Pic} \DeclareMathOperator{\Div}{Div}
\DeclareMathOperator{\Gal}{Gal} 
 
 \DeclareMathOperator{\rank}{rank}
\DeclareMathOperator{\Hom}{Hom} \DeclareMathOperator{\re}{Re}
\DeclareMathOperator{\im}{Im}   
\DeclareMathOperator{\vol}{vol} 
 \DeclareMathOperator{\Val}{Val}

\DeclareMathOperator{\Br}{Br} 

 \DeclareMathOperator{\res}{\partial}
 
 \DeclareMathOperator{\Sub}{Sub}
 \DeclareMathOperator{\Ind}{Ind}
\DeclareMathOperator{\HH}{H} \DeclareMathOperator{\Inn}{Inn}
\DeclareMathOperator{\RH}{R} \DeclareMathOperator{\Res}{R}

\DeclareMathOperator{\GL}{GL}
\DeclareMathOperator{\SL}{SL}
\DeclareMathOperator{\PGL}{PGL}
\DeclareMathOperator{\SO}{SO}
\DeclareMathOperator{\Spin}{Spin}
\DeclareMathOperator{\tr}{tr}

\let\det\relax
\DeclareMathOperator{\det}{det}

\DeclareSymbolFont{cyrletters}{OT2}{wncyr}{m}{n}
\DeclareMathSymbol{\Sha}{\mathalpha}{cyrletters}{"58}
\DeclareMathSymbol{\Be}{\mathalpha}{cyrletters}{"42}
\newcommand{\thorn}{\textit{\th}}

\newcommand{\OO}{\mathcal{O}}

\newcommand{\Norm}{N}

\newcommand{\Adele}{\mathbf{A}}

\newcommand{\sbf}{\mathbf{s}}
\newcommand{\A}{\mathscr{A}}
\newcommand{\br}{\mathscr{B}}
\newcommand{\R}{\mathscr{R}}
\newcommand{\X}{\mathscr{X}}

\newcommand\PP{\mathbb{P}}
\newcommand\ZZ{\mathbb{Z}}
\newcommand\NN{\mathbb{N}}
\newcommand\QQ{\mathbb{Q}}
\newcommand\RR{\mathbb{R}}
\newcommand\CC{\mathbb{C}}
\newcommand\GG{\mathbb{G}}

\newcommand\Gm{\GG_\mathrm{m}}

\newcommand{\K}{\mathsf K}

\newtheorem{lemma}{Lemma}
\newtheorem{proposition}[lemma]{Proposition}

\newtheorem{theorem}[lemma]{Theorem}
\newtheorem{corollary}[lemma]{Corollary}

\theoremstyle{definition}

\newtheorem{definition}[lemma]{Definition}
\newtheorem{remark}[lemma]{Remark}

\numberwithin{lemma}{section}

\title[Zero-loci of Brauer group elements on semi-simple algebraic groups]
{Zero-loci of Brauer group elements on semi-simple algebraic groups}
\author{\sc Daniel Loughran}
\address{
University of Manchester \\
School of Mathematics\\
Oxford Road\\
Manchester\\
M13 9PL\\
UK}
\email{daniel.loughran@manchester.ac.uk}

\author{\sc Ramin Takloo-Bighash}
\address{
Department of Mathematics, Statistics, and Computer Science\\
University of Illinois at Chicago\\
851 S Morgan St (M/C 249)\\
Chicago, IL 60202}
\email{rtakloo@math.uic.edu}

\author{\sc Sho Tanimoto}
\address{Department of Mathematics, Faculty of Science, Kumamoto University, Kurokami 2-39-1 Kumamoto 860-8555 Japan}
\email{stanimoto@kumamoto-u.ac.jp}



\begin{document}

\begin{abstract}
	We consider the problem of counting the number of rational points
	of bounded height in the zero-loci of Brauer group elements on
	semi-simple algebraic groups over number fields. We obtain asymptotic formulae for the counting
	problem for wonderful compactifications  using the spectral
	theory of automorphic forms. Applications include asymptotic formulae 
	for the number of matrices over $\QQ$
	whose determinant is a sum of two squares. These results provide a positive answer
	to some cases of a question of Serre concerning such counting problems.
\end{abstract}

\maketitle

\thispagestyle{empty}

\tableofcontents

\section{Introduction} \label{sec:intro}
Given a family of varieties over a number field $k$, a natural problem is to study the distribution
of the varieties in the family with a rational point. In general,
such problems are too difficult, due to needing to control failures of the Hasse principle.
Even in simple cases where the Hasse principle holds, such as for families of conics, these questions lead to rich
arithmetic problems.

Serre \cite{Ser90} considered such problems for families of conics parametrised
by a projective space, where he obtained upper bounds for the number of conics of bounded
height in the family with a rational point.
He showed, for example, that the number of integers  $|a_0|,|a_1|,|a_2| \leq B$ for which
the diagonal conic
\begin{equation} \label{eqn:diagonal}
	a_0x_0^2 + a_1x_1^2 + a_2x_2^2 = 0 \quad \subset \PP_\QQ^2
\end{equation}
has a rational point is $O(B^3/(\log B)^{3/2})$. He conjectured that his upper bounds were sharp,
and in the special case \eqref{eqn:diagonal} this was confirmed by Hooley \cite{Hoo93}.

In this paper we answer more cases of Serre's problem. One obtains a more general 
conceptual framework by working with \emph{Brauer group elements}, as opposed to families of conics; this approach
was also taken by Serre in \cite{Ser90}. One recovers the case of conics by working with Brauer group elements
coming from quaternion algebras (see \cite{Ser90, Lou13} for further details; we also give some applications of this type later
on).

The general set-up for such problems is as follows.
Let $X$ be a smooth projective variety over a number field $k$ with a height function $H$.
Let $U \subset X$ be a dense open subset and let $\br \subset \Br U$ be a finite subgroup of the Brauer group of $U$. Let
$$U(k)_\br=\{x \in U(k): b(x)=0 \in \Br k \text{ for all } b \in \br\},$$
denote the \emph{zero-locus} of $\br$. For $B>0$, we are interested in the counting function
\begin{equation} \label{def:Brauer_counting_function}
	N(U,H,\br,B) = \#\{x \in U(k)_\br:H(x) \leq B\}.
\end{equation}
Serre \cite{Ser90} obtained upper bounds in the case $X = \PP_\QQ^n$ and $\br$ has order $2$. Here, for example, on $\PP_\QQ^n$ one can take the height function
\begin{equation} \label{def:height_Pn}
	H(x_0,\ldots, x_n) = \sqrt{|x_0|^2 + \cdots + |x_n|^2},
\end{equation}
where we choose a representative such that $x_i \in \ZZ$ and $\gcd(x_0,\ldots,x_n) = 1$.
Serre's upper bounds were subsequently generalised to any finite collection
of Brauer group elements defined on any open subset of $\PP^n_k$ in \cite[Thm.~5.10]{LS16}, and state
that
$$N(U,H,\br,B) \ll \frac{B^{n+1}}{(\log B)^{\Delta_{\PP^n}(\br)}}, \quad \mbox{where }  \Delta_{\PP^n}(\br)=\sum_{D \in (\PP^n)^{(1)}}\left(1 - \frac{1}{|\res_D(\br)|}\right).$$
Here we denote by $X^{(1)}$ the set of codimension one points of a smooth variety $X$,  for $D \in X^{(1)}$ by $\res_D: \Br X \to \HH^1(\kappa(D), \QQ/\ZZ)$ the associated residue map, and by $|\res_D(\br)|$ the order of the group of residues $\res_D(\br)$.
These upper bounds are conjecturally sharp, and the problem is to obtain the correct lower bounds, or even
asymptotic formulae, under the necessary assumption that $U_{\br}(k) \neq \emptyset$.

A programme of study on such problems was began by the first-named author in \cite{Lou13}, where versions of Serre's
question were answered for anisotropic tori. For example, this gives results whenever $U \subset \PP_k^n$
is the complement of a norm form for a  field extension of $k$ of degree $n+1$. 
Another important case has been handled by Sofos \cite{Sof16},
which obtains the correct lower bounds whenever $U$ is the 
complement of a collection of closed points of $\PP_\QQ^1$
of total degree at most $3$ and $\br$ consists of a single element of order $2$.

Our first result answers Serre's question when $U$ is the complement
of the determinant hypersurface in $\PP_\QQ^{n^2-1}$.

\begin{theorem} \label{thm:PGL}
	Let $n>1$ and consider the determinant $\det_n: M_n(\QQ) \to \QQ$, viewed 
	as a homogeneous polynomial of degree $n$ in $n^2$ variables.
	Let $U = \PP_\QQ^{n^2-1} \setminus \{ \det_n(x) = 0\}$, equipped with the height $H$ from
	\eqref{def:height_Pn}. Let $b \in \Br U$ be an element
	such that $U(\QQ)_b \neq \emptyset$. Then there exists $c_{n,b,H} > 0$ such that
	$$N(U,H,b,B) \sim c_{n,b,H} \frac{B^{n^2}}{(\log B)^{1 - 1/|b|}}, \quad \text{ as } B \to \infty,$$
	where $|b|$ is the order of $b$ in the Brauer group.
\end{theorem}
\begin{remark}
For each rational point $P \in U(\mathbb Q)$, one may adjust $b$ by a constant algebra so that $P \in U(\QQ)_b$. In this way, the condition $U(\QQ)_b \neq \emptyset$ frequently happens as soon as we have $U(\QQ) \neq \emptyset$.
\end{remark}

This result has a classical interpretation in terms of Diophantine equations involving matrices
(note that the $U$ from Theorem \ref{thm:PGL} may be naturally identified with $\PGL_n$).
Namely, Theorem \ref{thm:PGL} allows one to count the number of matrices
of bounded height whose determinant is a norm from a given cyclic extension.

\begin{corollary} \label{cor:cyclic}
	Let $n > 1$ and let $K/\QQ$ be a cyclic extension of degree $d \mid n$. There exists $c_{n,K,H} > 0$ such that
	the number of matrices $g \in \PGL_n(\QQ) \subset \PP^{n^2-1}(\QQ)$ of height $H$ less than $B$ for which the equation
	$$\Norm_{K/\QQ}(x) = \det_n(g)$$
	has a solution for some $x \in K$, is asymptotically 
	$$c_{n,K,H} \frac{B^{n^2}}{(\log B)^{1 - 1/d}}, \quad B \to \infty.$$
\end{corollary}

The determinant of $g \in \PGL_n(\QQ)$ is a well-defined element $\det_n(g) \in \QQ^*/\QQ^{*n}$; in particular,
as $d \mid n$, asking whether $\det_n(g)$ is a norm from $K/\QQ$ is a well-posed question. 
One proves Corollary \ref{cor:cyclic} by applying Theorem \ref{thm:PGL} to the  element $b \in \Br \PGL_n$ given by the cyclic
algebra
$(\det_n, K/\QQ).$
This has order $d$, and the zero-locus consists of exactly
those $g \in \PGL_n(\QQ)$ for which $\det_n(g)$ is a norm from $K/\QQ$. Given that the identity matrix lies in $\PGL_n(\QQ)_b$,
one easily deduces Corollary \ref{cor:cyclic} from Theorem \ref{thm:PGL}. 

In the special case $K=\QQ(i)$, 
Corollary \ref{cor:cyclic} gives an asymptotic formula for the number of matrices whose determinant is a sum of two squares. This may  be viewed as an analogue for the determinant of the classical theorem of Landau \cite{Lan08} concerning the number of integers which can be written as a sum of two squares.

Our results for projective space are proved using more general results for zero-loci of Brauer group
elements on \emph{wonderful compactifications of adjoint semi-simple algebraic groups}. The wonderful compactification of an adjoint semi-simple algebraic group $G$ is a particularly nice bi-equivariant compactification of $G$, first introduced over algebraically closed fields by de Concini and Procesi in \cite{DCP83}. Favourable properties are that the boundary divisor is geometrically a strict normal crossing divisor, and that the effective cone of divisors is freely generated by the boundary divisors. We recall its construction and basic properties in \S\ref{sec:wonderful}. For technical reasons we work with height functions associated with a choice of \emph{smooth} adelic metric on the anticanonical bundle (see \S\ref{sec:heights} for details; the $(n+1)$th power of \eqref{def:height_Pn} on $\PP^n_\QQ$ is such a height function). Our result is as follows.

\begin{theorem} \label{thm:semi-simple}
	Let $G$ be an adjoint semi-simple algebraic group over a number field $k$ with wonderful compactification $G \subset X$.
	Let $H$ be a height function associated with a choice of smooth adelic metric on the anticanonical
	bundle of $X$.
	Let $\br \subset \Br_1 G$ be a finite subgroup of algebraic Brauer group elements. 
	Assume that $G(k)_\br \neq \emptyset$. Then there exists $c_{X,\br,H} > 0$ such that
	$$N(G,H,\br,B) \sim c_{X,\br,H} B \frac{(\log B)^{\rho(X)-1}}{(\log B)^{\Delta_X(\br)}},
	 \quad \text{as } B \to \infty,$$
   	where 
   	$$\Delta_X(\br)=\sum_{D \in X^{(1)}}\left(1 - \frac{1}{|\res_D(\br)|}\right),
   	\qquad \rho(X)=\rank \Pic X.$$
\end{theorem}

Here $\Br_1 G = \ker( \Br G \to \Br \overline{G})$ denotes the \emph{algebraic Brauer group} of $G$, where $\overline{G}$ is the base change of $G$ to the algebraic closure.
Theorem~\ref{thm:semi-simple} is proved using a similar strategy to the proof of Manin's conjecture for wonderful compactifications  \cite{JAMS} (this corresponds to the case $\br = 0$ in Theorem~\ref{thm:semi-simple}). We study the analytic properties of the associated height zeta function
$$\sum_{g \in G(k)_\br} \frac{1}{H(g)^s}$$
using harmonic analysis, namely the spectral theory of automorphic forms. This yields a spectral decomposition into parts coming from cuspidal, continuous, and $1$-dimensional automorphic representations (i.e.~automorphic characters of $G$). The leading singularity comes from the automorphic characters. The resulting height integrals are not meromorphic in general, but have \emph{branch point singularities}; this is reflected in the fact that the exponent of $\log B$ in Theorem \ref{thm:semi-simple} is a non-integral rational number, in general. Our proof has many parallels to the case of anisotropic toric varieties \cite{Lou13}.

The harmonic analysis approach is not  suited to transcendental (i.e.~non-algebraic) Brauer group elements
(see Remark \ref{rem:trans_bad}). This hypothesis was also necessary in the case of anisotropic tori
\cite{Lou13} (note however, contrary to \emph{loc.~cit.}, we are able to deal with arbitrary
adjoint groups, rather than just anisotropic groups).

For an absolutely simple adjoint semi-simple algebraic group $G$, the transcendental Brauer group is non-trivial only if $G$
has type $D_{2n}$ (Remark \ref{rem:abs_simple}). In particular, Theorem~\ref{thm:semi-simple}
gives a complete answer to Serre's question for most such $G$ (e.g.~for $\PGL_n$).

Many of the results in the literature on Serre's problem (e.g.~\cite{Hoo93}, \cite{Hoo07}, and \cite{Sof16}) only give lower bounds for the counting problem. On the other hand, in Theorem \ref{thm:semi-simple} we are able to obtain a precise asymptotic formula for the counting problem. We also obtain in Theorem \ref{thm:leading_constant} an explicit expression for the leading constant $c_{X,\br,H}$ in terms of a certain Tamagawa measure.

Note that Theorem \ref{thm:PGL} does not follow immediately from Theorem~\ref{thm:semi-simple},
as the given compactification $\PGL_n \subset \PP^{n^2-1}$ is not wonderful when $n>2$ (the boundary divisor is geometrically irreducible but singular, hence not strict normal crossings). We prove Theorem~\ref{thm:PGL} using the explicit construction of the wonderful compactification of $\PGL_n$ as a blow up $X \to \PP^{n^2-1}$, and then use the functoriality of heights to relate the counting problem to one on the wonderful compactification. To resolve the corresponding counting problem we require a more general version of Theorem \ref{thm:semi-simple} for counting rational points of bounded height with respect to height functions attached to arbitrary big line bundles. This is Theorem \ref{thm:arbitrarybiglinebundles}, which is the main theorem in this paper.

The layout of the paper is as follows. In \S \ref{sec:Brauer} we study Brauer groups of semi-simple algebraic groups. Our main result here is Theorem \ref{thm:Br_e(G)}, which gives a description of the Brauer group of a semi-simple algebraic group of a number field in terms of its automorphic characters. In \S \ref{sec:wonderful} we recall various properties of the wonderful compactification. In \S\ref{sec:main} we prove Theorem \ref{thm:semi-simple}, together with its generalisation to big line bundles (Theorem \ref{thm:arbitrarybiglinebundles}). In \S\ref{sec:PGLn} we prove Theorem \ref{thm:PGL} using Theorem \ref{thm:arbitrarybiglinebundles}.

\subsection{Notation and conventions} \label{sec:notation}
For a topological group $G$, we denote by $G^{\wedge} = \Hom(G, S^1)$ its group of continuous unitary characters and by $G^\sim = \Hom(G, \QQ/\ZZ)$ its group of continuous $\QQ/\ZZ$-characters.  We choose an embedding $\QQ/\ZZ \subset S^1$, so that 
$G^\sim \subset G^\wedge$. We commit the following abuse of notation: For
a closed (not necessarily normal) subgroup $H \subset G$, we denote by $(G/H)^\wedge$
the collection of characters of $G$ which are trivial on $H$. We use the notation $(G/H)^\sim$
analogously.

For a finite group $G$, we denote its order by $|G|$. For an element $g \in G$, we also denote by $|g|$ its order.

Let $D \subset \CC$ be a subset and $f:D \to \CC$. We say that $f$ is \emph{holomorphic} on $D$ if there exists an open subset $D \subset U \subset \CC$ and a holomorphic function $g: U \to \CC$ such that $g|_D = f$.

Let $U \subset \CC$ be a connected open subset and $a \in \CC$ a point on the boundary of $U$. Let $f: U \to \CC$ be holomorphic and let $q \in \{ q \in \RR: q \notin \ZZ_{\geq 0}\}$.
We say that $f$ admits a \emph{branch point singularity of order $q$ at $a$} 
if $\lim_{s \to a}(s-a)^{-q} f(s)$ exists and is non-zero.
Here we interpret $(s-a)^{-q}$ with its usual branch cut, namely as a holomorphic function on $\CC \setminus \{\sigma \in \RR: \sigma \leq a \}$.

All cohomology is taken with respect to the \'{e}tale topology. For a smooth variety $X$ over a field $k$, 
we denote by $\overline{X} = X \otimes_k \overline{k}$ the base-change to a fixed choice of algebraic closure $\overline{k}$ of $k$
and by $\Br X = \HH^2(X,\Gm)$ its  Brauer group. We denote by $ \Br_1 X = \ker( \Br X \to \Br \overline{X})$ the algebraic part of the Brauer group of $X$. We say that an  element of $\Br X$ is \emph{constant} if it lies in the image of the map $\Br k \to \Br X$ (this map need not be injective in general, however it is injective
if $X(k) \neq \emptyset$).
An element of $\Br X$ which does not lie in $\Br_1 X$ is called \emph{transcendental}.

Let $k$ be a number field. We denote by $\Val(k)$ the set of places of $k$ and $k_v$ the completion at a place $v$; if $v$ is non-archimedean $\OO_v$ denotes the ring of integers of $k_v$. Let $G$ be a linear algebraic group over $k$ or a $\Gal(\overline{k}/k)$-module. Then we denote by
$$\Sha(k,G) = \ker\left( \HH^1(k, G) \to \prod_{v \in \Val(k)} \HH^1(k, G)\right)$$
the Tate--Shafarevich set of $G$. (If $G$ is non-abelian we use non-abelian \v{C}ech cohomology \cite[p.~122]{Mil80}).

\subsection*{Acknowledgements}
We are very grateful to Martin Bright for help with the proof of Proposition \ref{prop:residue_restriction}, to Giuliano Gagliardi for useful discussions on wonderful compactifications, and to Alex Gorodnik and Yiannis Sakellaridis for helpful conversations regarding matrix coefficients. We  thank Dylon Chow who read a draft of the paper and made several helpful comments. We also thank the referee for many useful comments and suggested improvements to the exposition of the paper. Ramin Takloo-Bighash is partially funded by a Simons Foundation Collaborative Grant (Award number 245977).  Sho Tanimoto is partially supported by Lars Hesselholt's Niels Bohr Professorship and MEXT Japan, Leading Initiative for Excellent Young Researchers (LEADER).

\section{Brauer groups} \label{sec:Brauer}
In this section we gather results on the Brauer groups of semi-simple algebraic groups.

\subsection{Generalities}

Let $G$ be a semi-simple algebraic group over a field $k$ of characteristic $0$ with identity
element $e \in G(k)$.

\subsubsection{Calculating the Brauer group}
We shall be primarily  interested in the algebraic Brauer group $\Br_1 G$ of $G$. 
By translating by an element of $\Br k$, it suffices to study the subgroup
\begin{equation} \label{def:Br_e}
	\Br_e G = \{ b \in \Br_1 G : b(e) = 0 \in \Br k\}.
\end{equation}

\begin{lemma} \label{lem:Br}
	There is a functorial isomorphism $$\Br_e G \cong \mathrm{H}^1(k, \Pic \overline{G}).$$
	We also have an isomorphism
	$$\Br \overline{G} \cong \mathrm{H}^3(\Pic \overline{G},\ZZ).$$
	In particular $\Br \overline{G} = 0$ if and only if $\Pic \overline{G}$ is cyclic.
\end{lemma}
\begin{proof}
	The first part is \cite[Lem.~6.9(iii)]{San81}. This is proved using the Hochschild-Serre spectral sequence; the functoriality
	of the isomorphism therefore follows from the functoriality of this spectral sequence. The second part is a result of Iversen \cite[Cor.~4.6]{Ive76} (see also \cite[Rem.~4.7]{Ive76}). For the last part, see \cite[Cor.~4.3]{Ive76}. 	
\end{proof}

Note that this result implies that the adjoint case is in many respect the most interesting case 
(e.g.~simply connected semi-simple algebraic groups have constant Brauer group).

\begin{remark} \label{rem:abs_simple}
Let $G$ be a split adjoint absolutely simple semi-simple algebraic group. Then we have the following possibilities
for the Picard group of $G$ (this can be obtained from \cite[Tab.~9.2.]{MT11}).
\begin{table}[ht]
\centering
$\begin{array}{|l|l|l|l|l|l|}
	\hline
	\textrm{Type} & A_n \,(n \geq 1) & B_n,C_n, E_7 \,(n \geq 2) & D_{2n+1} \, (n \geq 1) & D_{2n} \, (n \geq 2)  & G_2,F_4,E_8 \\		
	\hline
	\Pic G & \ZZ/(n+1)\ZZ & \ZZ/2\ZZ & \ZZ/4\ZZ & \ZZ/2\ZZ \times \ZZ/2\ZZ & 0 \\ \hline	
\end{array}$
\end{table} 

\noindent Using Lemma \ref{lem:Br}, one finds that only groups of type 
$D_{2n}$ may have non-trivial transcendental Brauer group, 
that groups of type $G_2,F_4,E_8$ have constant Brauer group, and that the remaining types have non-constant Brauer group in general. Of course more possibilities arise for non-simple groups, as 
$\Pic(G_1 \times G_2) = \Pic(G_1) \times \Pic(G_2)$.
\end{remark}

\subsubsection{The Brauer pairing}
Recall that one can evaluate an element $b$ of the Brauer group of a variety $X$ over $k$
at a rational point $x \in X(k)$ to obtain an element $b(x) \in \Br k$.
We focus on algebraic Brauer groups, as in this case the induced pairing is much better behaved on $G$, due 
to the following result of Sansuc \cite[Lem.~6.9(i)]{San81}.

\begin{lemma} \label{lem:bilinear}
	The pairing
	$$\Br_e G \times G(k) \to \Br k, \quad (b,g) \mapsto b(g)$$
	is bilinear.
\end{lemma}

\subsubsection{Residues} \label{sec:residues}
Let $G \subset X$ be a smooth projective compactification of $G$ with boundary divisor $D = X \setminus G$. We have the exact sequence \cite[(9.0.2)]{San81}
\begin{equation} \label{eqn:Div}
	0 \to \Div_{\overline{D}}\overline{X} \to \Pic \overline{X} \to \Pic \overline{G} \to 0
\end{equation}
where $\Div_{\overline{D}}\overline{X}$ denotes the group of divisors of $\overline{X}$ which are supported on $D$. Applying Galois cohomology,
we obtain the map
$$
	\HH^1(k, \Pic \overline{G}) \to \HH^2(k, \Div_{\overline{D}}\overline{X}).
$$
However, the group $\Div_{\overline{D}}\overline{X}$ is a permutation module. In particular, applying Shapiro's lemma and using Lemma \ref{lem:Br}
we obtain a map
\begin{equation} \label{eqn:residue}
	\res: \Br_e G \to \hspace{-10pt} \prod_{\alpha \in X^{(1)} \cap D} \hspace{-10pt} \HH^1(k_\alpha, \QQ/\ZZ),
\end{equation}
where $k_\alpha$ denotes the algebraic closure of $k$ in the residue field of $\alpha$. A result of Sansuc \cite[Lem.~9.1]{San81}
states that the maps from \eqref{eqn:residue} agree with the usual residue maps $\Br_e G \to \HH^1(k_\alpha, \QQ/\ZZ)$ \cite[\S 6.8]{Poo17}, up to sign.
We take \eqref{eqn:residue} to be the \emph{definition} of the residue map attached to $G \subset X$. This possible sign will not cause problems, as ultimately in Theorem \ref{thm:semi-simple} we only care about the subgroup generated by the residues of  $\br$, which is unaffected by a sign change.

\subsubsection{Inner forms and transfers} \label{sec:transfer}

Denote by $\Inn G$ the inner automorphism group of $G$; this is naturally an adjoint semi-simple algebraic group over $k$ which acts on $G$ by conjugation. We can twist by (a $1$-cocycle representing) an  element of $\HH^1(k, \Inn G)$ to obtain a so-called \emph{inner form} of $G$. There is a unique (up to isomorphism) quasi-split inner form of $G$; we denote this by $G'$.

To assist with the harmonic analysis, we will need to perform a transfer between the Brauer group of $G$ and $G'$. Recall from Lemma~\ref{lem:Br} that we have a canonical
isomorphism
$$\Br_e G \cong \HH^1(k, \Pic \overline{G}).$$
However $G$ and $G'$ differ by an inner twist, and inner automorphisms act trivially on the Picard group as $\Inn G$ is connected.
It follows that there is a canonical
isomorphism
\begin{equation} \label{eqn:Br_transfer}
	\tr: \Br_e G \to \Br_e G',
\end{equation}
which we call the \emph{transfer map}.
We show that this is compatible with taking residue maps with respect to \emph{bi-equivariant compactifications},
i.e.~those compactifications $G \subset X$ for which both the natural left and right action of $G$ on itself extends to $X$. Note that $\Inn G$ naturally acts on such a compactification, so we can twist by the relevant cohomology class to obtain a natural bi-equivariant compactifiction $G' \subset X'$.
	
\begin{lemma} \label{lem:transfer_residues}
	Let $G \subset X$ be a smooth projective bi-equivariant compactification of $G$ with boundary divisor $D$.
	Let $G'$ be the quasi-split inner form of $G$ with associated bi-equivariant compactification 
	$G' \subset X'$ and boundary divisor $D'$. Then the transfer for the Brauer group respects the residue maps, i.e.~we have a commutative diagram
	$$
	\xymatrix{
		 \Br_e G \ar[d] \ar[r] & \Br_e G' \ar[d] \\
		 \prod_{\alpha \in X^{(1)} \cap D} \HH^1(k_\alpha, \QQ/\ZZ) \ar[r] & \prod_{\alpha' \in X^{'(1)} \cap D'} \HH^1(k_{\alpha'}, \QQ/\ZZ).}
	$$
\end{lemma}
\begin{proof}
	As $\Inn G$ is connected it 
	acts trivially on the boundary divisors and on $\Pic \overline{G}$. We thus find that
	$$0 \to \Div_{\overline{D}}\overline{X} \to \Pic \overline{X} \to \Pic \overline{G} \to 0 $$
	and
	$$0 \to \Div_{\overline{D}'}\overline{X}' \to \Pic \overline{X}' \to \Pic \overline{G}' \to 0 $$
	are canonically isomorphic as exact sequence of Galois modules. The result then follows from Sansuc's
	description of the residue map given in \eqref{eqn:residue}.
\end{proof}

\subsection{Brauer groups over number fields}
Let now $G$ be a semi-simple algebraic group over a number field $k$. 
We fix an integral model of $G$ over $\mathcal O_k$ where $\mathcal O_k$ is the ring of algebraic integers for $k$.
Taking the sum of the local pairings over each completion $k_v$ of $k$,
combined with the embeddings $\Br k_v \subset \QQ/\ZZ$,
we obtain a global Brauer pairing
\begin{equation} \label{def:global_Brauer_pairing}
	\Br G \times G(\Adele_k) \to \QQ/\ZZ,
\end{equation}
which is right continuous \cite[Cor.~8.2.11]{Poo17}.

By an \emph{automorphic character} of $G$,
we mean a continuous homomorphism $G(\Adele_k) \to S^1$ which is trivial on $G(k)$. For such a character $\chi$, we denote by $\chi_v$ its local component at a place $v$.  Following
our conventions from \S \ref{sec:notation}, we denote the group of automorphic characters of $G$ by
$(G(\Adele_k) / G(k))^\wedge$. Note that there is little difference between this and the group of $\QQ/\ZZ$-characters $(G(\Adele_k) / G(k))^\sim$, as every element of $(G(\Adele_k) / G(k))^\wedge$ has finite order (in fact there exists $n \in \NN$, depending only on $G$, such that $\chi^n = 1$ for all $\chi \in (G(\Adele_k) / G(k))^\wedge$, cf.~\cite[Rem.~2.2]{JAMS}).

\begin{lemma} \label{lem:Br_char}
	Let $b \in \Br_e G$.
	The map
	$$G(\Adele_k) \to \QQ/\ZZ \subset S^1,$$
	induced by pairing with $b$, is an automorphic character of $G$.
\end{lemma}
\begin{proof}
	That it is continuous follows from the continuity of the Brauer pairing. 
	That the induced map is a homomorphism follows from Lemma \ref{lem:bilinear}. 
	Its triviality on  $G(k)$ follows from the exact sequence 
	\begin{equation} \label{seq:CFT}
		0 \to \Br k \to \, \prod_{\mathclap{v \in \Val(k)}} \, \Br k_v \to \QQ/\ZZ \to 0
	\end{equation}
	from class field theory \cite[Thm.~1.5.36]{Poo17}.
\end{proof}

We use this to obtain the following, which is crucial for the harmonic analysis.
Let $\thorn_{\br,v}: G(k_v) \to \{0,1\}$ be the indicator function of the zero locus $G(k_v)_\br$
of $\br$ in $G(k_v)$.

\begin{lemma} \label{lem:K-invariant}
	Let $\br \subset \Br_e G$ be a finite subgroup and $v$ a place of $k$.
	\begin{enumerate}
		\item For any place $v$ of $k$, the indicator function $\thorn_{\br,v}$ is locally constant.
		\item For any nonarchimedean place $v$ of $k$, the indicator function $\thorn_{\br,v}$ 
		is bi-invariant under some compact open subgroup  $\K_v \subset G(k_v)$.
		Moreover, one may take $\K_v = G(\OO_v)$ for all but finitely many non-archimedean $v$.
	\end{enumerate}
\end{lemma}
\begin{proof}
	Part (1) follows from the right continuity of the Brauer pairing $\Br G_v \times G(k_v) \to \QQ/\ZZ$.
	Let $\R$ be the group of automorphic characters attached to $\br$ by Lemma  \ref{lem:Br_char}. Note that $\thorn_{\br,v}$ is the indicator function of $\cap_{\rho \in \R} \ker \rho_v$. 
	Character orthogonality
	yields
	\begin{equation} \label{eqn:thorn_rho}
		\thorn_{\br,v}(g_v) = \frac{1}{|\R|}\sum_{\rho \in \R} \rho_v(g_v), \quad \mbox{for all } g_v \in G(k_v).
	\end{equation}
	As each $\rho_v$ has finite order, the existence of $\K_v$ easily follows from \eqref{eqn:thorn_rho}.
	Moreover, as each $\rho$ is an automorphic character, for all but finitely many $v$
	we see that $\rho_v$ is trivial	on $G(\OO_v)$, whence the result.
\end{proof}

\begin{remark} \label{rem:trans_bad}
	The conclusion of Lemma \ref{lem:K-invariant} fails for transcendental Brauer group elements in general.
	This is why we focus on the algebraic Brauer group in this paper; the harmonic analysis tools
	are not suited for transcendental Brauer group elements.
	
	Consider $G=\PGL_2 \times \PGL_2$ over $\QQ$ with the quaternion algebra
	$$b=(\det_2, \det_2)$$
	on $G$. A simple Hilbert symbol calculation shows that $\thorn_{b,p}$ is  bi-invariant under the image
	of $\SL_2(\ZZ_p) \times \SL_2(\ZZ_p)$ in $G(\QQ_p)$ for all odd primes $p$,
	but not bi-invariant under $\PGL_2(\ZZ_p) \times \PGL_2(\ZZ_p)$ for any prime $p$.
	A similar phenomenon occurred in the case of algebraic tori (see 
	\cite[Rem.~5.5]{Lou13}).
\end{remark}

\subsubsection{Calculating the Brauer group}
We now calculate the algebraic Brauer group over number fields. To do so, we require the following
theorem due to Chernousov, Harder and Kneser (see \cite[Thms.~6.4, 6.6]{PR94}).

\begin{theorem} \label{thm:Kneser}
	Let $G^{sc}$ be a simply connected semi-simple algebraic group over $k$. Then for non-archimedean $v$ we have
	$\HH^1(k_v, G^{sc}) = 0.$
	Moreover, the natural map
	$$ \HH^1(k, G^{sc}) \to \prod_{v \mid \infty} \HH^1(k_v,G^{sc})$$
	is a bijection.
\end{theorem}

We combine this with an application of Poitou--Tate duality to obtain the following.
An analogous result for algebraic tori can be found in \cite[Thm 4.5]{Lou13}.
In the statement of the theorem,  we denote by
$$\Be(G) = \ker\left(\Br_e G \to \, \prod_{\mathclap{v \in \Val(k)}} \, \Br_e G_v\right)$$
(here $\Be$ is the Cyrillic Be). It is well known that $\Be(G)$ is finite  (see e.g.~\cite[Prop.~9.8]{San81}).

\begin{theorem} \label{thm:Br_e(G)}
	The Brauer pairing
	$$\Br_e G \times G(\Adele_k) \to \QQ/\ZZ$$
	induces a short exact sequence 
	$$ 0 \to \Be(G) \to \Br_e G \to (G(\Adele_k) / G(k))^\sim \to 0.$$
\end{theorem}
\begin{proof}
	\textbf{Step $0$: Generalities:}
	Let $G^{sc} \to G$ be the simply connected cover of $G$, with scheme theoretic kernel $F$,
	so that we have the exact sequence of group schemes
	\begin{equation} \label{eqn:ZGG}
		1 \to F \to G^{sc} \to G \to 1.
	\end{equation}
	Note that $F$ is the Cartier dual of the Picard scheme of $G$ (see \cite[Lem.~6.9(iii)]{San81}). 
	
	For any field extension $k \subset L$ we will also require the diagram
	\begin{equation}\label{diag:Br_commute}
	\begin{split}
	\xymatrix@C=0.1pt{
		\Br_e G_L &\times 		\ar[d] 				& G(L)  \ar[rrrrr] &  & & & & \Br L \ar@{=}[d] \\
		\HH^1(L, \Pic \overline{G})  & \times &  \HH^1(L, F)  \ar[rrrrr]^{\,\,\,\, \smile} & & & & & \Br L . }
	\end{split}
	\end{equation}	
	The top arrow is the Brauer pairing, 
	the left hand vertical arrow is the isomorphism from Lemma \ref{lem:Br}, the right hand vertical arrow
	comes from the sequence \eqref{eqn:ZGG}, and the bottom arrow is the cup-product.
	By \cite[Lem.~8.11]{San81}, this diagram is anticommutative.
	
	\textbf{Step $1$: Local fields:}
	Let $v$ be a \emph{non-archimedean} place of $k$. Applying Galois cohomology to \eqref{eqn:ZGG} we obtain
	the exact sequence
	\begin{equation} \label{eqn:ZGG_v}
		1 \to F(k_v) \to G^{sc}(k_v) \stackrel{j_v}{\to} G(k_v) \to \HH^1(k_v, F) \to 1,
	\end{equation}
	of pointed topological spaces, where we have used the vanishing $\HH^1(k_v, G^{sc})=0$ of Theorem \ref{thm:Kneser}.
	Note that $j_v(G^{sc}(k_v)) \subset G(k_v)$ is a closed normal subgroup of finite index and the topological
	quotient $G(k_v)/j_v(G^{sc}(k_v))$ has the discrete topology.
	Local Tate duality \cite[Thm.~7.2.6]{NSW08} implies that the cup-product pairing
	\begin{equation} \label{eqn:Tate}
		\HH^1(k_v, \Pic \overline{G}) \times \HH^1(k_v, F) \to \QQ/\ZZ
	\end{equation}
	is perfect. It follows from this, \eqref{diag:Br_commute}, \eqref{eqn:ZGG_v}, and the above topological considerations,
	that the pairing
	$$\Br_e G_v \times G(k_v) \to \QQ/\ZZ$$
	induces an isomorphism 
	\begin{equation} \label{eqn:Br_v}
		\Br_e G_v \cong (G(k_v)/j_v(G^{sc}(k_v)))^\sim	,
	\end{equation}
	for any non-archimedean place $v$.
	
	\smallskip	
	\textbf{Step $2$: The kernel:}
	Next, consider the homomorphism
	$$\varepsilon: \Br_e G \to G(\Adele_k)^\sim$$
	from Lemmas \ref{lem:bilinear} and \ref{lem:Br_char}.  Clearly $\Be(G) \subset \ker \varepsilon$. To prove the converse, let $b \in \ker \varepsilon$.
	By definition, we see that $b$ induces the trivial character at all places $v$. From \eqref{eqn:Br_v}, we have $b_v = 0 \in \Br_e G_v$
	for all non-archimedean places $v$. Hence from Lemma \ref{lem:Br} we deduce that, with respect to the isomorphism $\Br_e G \cong \HH^1(k, \Pic \overline{G})$,
	we have
	$$\ker \varepsilon \subset \Sha_{\infty}(k,\Pic \overline{G}) := \ker\left( \HH^1(k, \Pic \overline{G}) \to \prod_{v \nmid \infty} \HH^1(k, \Pic \overline{G})\right).$$
	However, as the decomposition group at an archimedean place is cyclic, it follows from \cite[Lem.~1.1]{San81} that we in fact have 
	$\Sha(k,\Pic \overline{G}) = \Sha_{\infty}(k,\Pic \overline{G}).$ Another result of Sansuc \cite[Cor.~7.4]{San81} yields
	$\Be(G) \cong \Sha(k,\Pic \overline{G})$, thus $|\ker \varepsilon| \leq |\Be(G)|$ and so $\ker \varepsilon = \Be(G)$, as claimed.
	
	\smallskip	
	\textbf{Step $3$: Surjectivity:}
	We finish by showing the required surjectivity using the Poitou--Tate exact sequence (see for example $(8.6.10)$ in \cite{NSW08}).
	As usual, for a finite abelian group scheme $M$ over $k$ we let 
	$$\mathrm{P}^1(k, M) = {\prod_v}' \HH^1(k_v, M)$$
	be the restricted direct product with respect to the subgroups $\HH^1(\OO_{v}, M)$ (these subgroups are well-defined for all but finitely many non-archimedean $v$).
	This has the topology of a locally compact abelian group.
	The part of Poitou--Tate relevant to us is
	$$\HH^1(k, M) \to \mathrm{P}^1(k,M) \to \HH^1(k, \widehat{M})^\sim,$$
	which is an exact sequence of topological groups.
	We apply this with $M=F$, and use Lemma \ref{lem:Br} to deduce that
	the natural map
	\begin{equation} \label{eqn:PT}
		\mathrm{P}^1(k,F)/\HH^1(k,F) \to (\Br_e G)^\sim
	\end{equation}
	is injective and that its image has the subspace topology 
	(here we abuse notation slightly, as the map $\HH^1(k, M) \to \mathrm{P}^1(k,M)$ need not be injective in general).
	
	We now apply \'{e}tale cohomology to  \eqref{eqn:ZGG}. Recall that the restricted direct product of exact sequences is again exact (as direct products and direct limits preserve exact sequences). By Theorem \ref{thm:Kneser}
	we therefore obtain the commutative diagram
	\begin{equation}
	\begin{split} \label{eqn:Intesection}
	\xymatrix{
		 G^{sc}(\Adele_k) \ar[r]^j & G(\Adele_k) \ar[r] & \mathrm{P}^1(k,F) \ar[r] & \prod_{v \mid \infty} \HH^1(k_v, G^{sc}) \\
		 G^{sc}(k) \ar[r] \ar[u] & G(k) \ar[r] \ar[u] & \HH^1(k,F) \ar[r] \ar[u] & \HH^1(k, G^{sc}) \ar[u]
	}
	\end{split}
	\end{equation}
	whose rows are exact sequences of pointed topological spaces. By Theorem \ref{thm:Kneser}, the rightmost arrow is a bijection.
	Using this, a chase through \eqref{eqn:Intesection} shows that
	$$\im (G(\Adele_k) \to \mathrm{P}^1(k,F)) \cap \im (\HH^1(k,F) \to \mathrm{P}^1(k,F)) = \im(G(k) \to  \mathrm{P}^1(k,F)).$$
	As $\HH^1(k_v, G^{sc})$ is finite for $v \mid \infty$, we thus see that 
	$$G(\Adele_k)/j(G^{sc}(\Adele_k))G(k) \subset \mathrm{P}^1(k,F)/\HH^1(k,F),$$
	is an open subgroup of finite index.
	Hence, using \eqref{eqn:PT}, the natural map
	$$G(\Adele_k)/j(G^{sc}(\Adele_k))G(k) \to (\Br_e G)^\sim$$
	is an injection whose image is equipped with the subspace topology. Applying Pontryagin duality, we deduce that the map
	$$\Br_e G \to (G(\Adele_k)/j(G^{sc}(\Adele_k))G(k))^\sim$$	
	is surjective. Thus to complete the proof, it suffices to show that the map
	\begin{equation} \label{eqn:top_iso}
		(G(\Adele_k)/j(G^{sc}(\Adele_k))G(k))^\sim \to (G(\Adele_k)/G(k))^\sim
	\end{equation}
	is an isomorphism of topological groups. It is clearly injective. To show surjectivity, let $\chi \in (G(\Adele_k)/G(k))^\sim$.
	As $G^{sc}$	is simply connected, any automorphic character of $G^{sc}$ is trivial \cite[Prop.~2.1]{JAMS}.
	Thus the restriction of $\chi$ to $j(G^{sc}(\Adele_k))$ is also trivial, and so
	\eqref{eqn:top_iso} is an isomorphism, as required.
\end{proof}

\begin{remark}
	The analogue of \eqref{eqn:Br_v} is \emph{false} for real places in general. For a real place $v$, the map
	$$\Br_e G_v \to (G(k_v)/j_v(G^{sc}(k_v)))^\sim	,$$
	is surjective, but need not be injective. For example, consider $\SO_n$ over $\RR$. From Lemma \ref{lem:Br} we have
	$\Br_e \SO_n \cong \mu_2$, but the map
	$ \Spin_n(\RR) \to \SO_n(\RR)$ is surjective. Of course the problem here is that the analogue of Theorem \ref{thm:Kneser}
	does not hold; the set $\HH^1(\RR, \Spin_n)$ is non-trivial.
\end{remark}

\begin{corollary} \label{cor:Br_adjoint}
	Suppose that $G$ is adjoint. Then the map $$\Br_e G \to (G(\Adele_k) / G(k))^\sim$$
	is an isomorphism.
\end{corollary}
\begin{proof}
	It suffices to note that $\Be(G) = 0$ in this case (see \cite[Prop.~9.8]{San81}).
\end{proof}

\subsubsection{Compatibility of transfers} \label{sec:transfer_Br_Ch}

In the case that $G$ is adjoint and $k$ is a number field, a transfer map 
$$\tr: (G(\Adele_k)/G(k))^\sim \to (G'(\Adele_k)/G'(k))^\sim $$
on automorphic characters was constructed in \cite[\S 2]{JAMS}, where $G'$ denotes the quasi-split inner form
of $G$
(we recall this construction in the proof of Lemma~\ref{lem:transfer_characters}).
We next show that this transfer is compatible with our transfer, as defined in \S\ref{sec:transfer}.

\begin{lemma} \label{lem:transfer_characters}
	Assume that $G$ is adjoint with quasi-split inner form $G'$.
	The transfer map on Brauer groups and automorphic characters commute, i.e.~they give rise to a commutative
	diagram
	$$
	\xymatrix{
		 \Br_e G \ar[d]^{\tr} \ar[r] & (G(\Adele_k)/G(k))^\sim \ar[d]^{\tr} \\
		 \Br_e G' \ar[r] & (G'(\Adele_k)/G'(k))^\sim.}
	$$
\end{lemma}
\begin{proof}
	The transfer on automorphic characters is built from a collection of local transfers.
	We first recall the definition of these local transfer for a non-archimedean place $v$,
	as defined in \cite[\S 2.2]{JAMS}. As explained in the proof of Theorem~\ref{thm:Br_e(G)}
	(see \eqref{eqn:Br_v}),
	the natural map
	$$ \HH^1(k_v, \Pic \overline{G}) \to (G(k_v)/j_v(G^{sc}(k_v)))^\sim$$
	is an isomorphism. The left-hand side does not change upon twisting by an inner automorphism, hence
	we obtain an isomorphism
	$$\tr_v: (G(k_v)/j_v(G^{sc}(k_v)))^\sim \to (G'(k_v)/j'_v(G^{'sc}(k_v)))^\sim$$
	which is the definition of the local transfer from \cite[\S 2.2]{JAMS}. In the light of Lemma \ref{lem:Br},
	we see that the local transfer $\tr_v$ and the transfer map $\Br_e G_v \to \Br_e G'_v$ on Brauer
	groups commute for all non-archimedean places $v$.
	
	We have shown that the automorphic characters agree at all non-archimedean places. 
	This in fact implies that they agree at all places.
	Indeed, if $\chi$ is an automorphic character of $G$ whose local components $\chi_v$ are trivial for all
	but finitely many $v$, then a simple application of weak approximation for $G$ 
	implies that $\chi$ is actually trivial.
\end{proof}

\subsection{The case of $\PGL_n$}

We now explain the above theory in the special case where $G = \PGL_n$. We have the following explicit
description of the Brauer group in terms of cyclic algebras 
(see \cite[\S1.5.7]{Poo17} for a general treatment of cyclic algebras).

\begin{proposition} \label{prop:PGL_n}
	Let $n \in \NN$ and let $k$ be a field of characteristic $0$.
	Then
	$$\Br \overline{\PGL_n} = 0, \quad \Br_e \PGL_n \cong \Hom(\Gal(\overline{k}/k), \ZZ/n\ZZ).$$
	Explicitly, every element of $\Br_e \PGL_n$ has a representative by a cyclic
	algebra of the form $(\det_n, \alpha)$ for some
	$\alpha \in \Hom(\Gal(\overline{k}/k), \ZZ/n\ZZ)$, where $\det_n$ denotes the determinant on $\PGL_n \subset \PP_k^{n^2-1}$
	viewed as a homogeneous polynomial of degree $n$. 
\end{proposition}
\begin{proof}
	As $\Pic \overline{\PGL}_n \cong \ZZ/n\ZZ$, Lemma \ref{lem:Br} implies that
	the transcendental Brauer group is trivial. One could also use Lemma \ref{lem:Br}
	for the explicit description of the algebraic
	Brauer group, but we give a direct proof instead.
	We use the exact sequence \eqref{eqn:Div} with respect to the compactification $\PGL_n \subset \PP_k^{n^2-1}$.
	This sequence here is just
	$$ 0 \to \ZZ \to \ZZ \to \ZZ/n\ZZ \to 0.$$
	Applying Galois cohomology, we obtain the exact sequence 
	\begin{equation} \label{eqn:res_PGL_n}
		0 = \HH^1(k,\ZZ) \to \Br_e \PGL_n \to \HH^2(k,\ZZ)[n] \to 0.
	\end{equation}
	The residue map $\res: \Br_e \PGL_n \to \HH^1(k,\QQ/\ZZ)$ is given by composing this map
	with the canonical isomorphism $\HH^2(k,\ZZ)[n] \cong \HH^1(k,\QQ/\ZZ)$, thus
	$\res$ is an isomorphism.
	However, the residue of the cyclic algebra $(\det_n, \alpha)$ is simply $\alpha$;
	the result follows.
\end{proof}

When $k$ is a number field, under the isomorphism of Corollary \ref{cor:Br_adjoint} the cyclic algebra
$(\det_n, \alpha)$ corresponds to the automorphic $\QQ/\ZZ$-character
$$ \PGL_n(\Adele_k) \to \ZZ/n\ZZ, \quad  g \mapsto \chi(\det_n g),$$
where $\chi: \Adele_k^* \to \ZZ/n\ZZ$ is the automorphic character attached to $\alpha$ via class field theory.

\section{Wonderful compactifications} \label{sec:wonderful}

In this section we recall the construction of wonderful compactifications for adjoint semi-simple groups
and their basic properties.
The wondeful compactification was introduced by de Concini and Procesi in \cite{DCP83} over algebraically closed fields of characteristic $0$.
We work over an arbitrary field $k$ of characteristic $0$, and explain the construction in this setting.

Let $G$ be an adjoint semi-simple algebraic group over $k$. There are many ways to construct the wonderful
compactification of $G$. We use the following, as it is clear that it works over non-algebraically closed fields and 
avoids the choice of  a maximal torus.

\begin{definition}
	Let $\mathfrak{g}$ be the Lie algebra of $G$ and $\mathbb{L}$ the variety of Lie subalgebras
	of $\mathfrak{g} \oplus \mathfrak{g}$, viewed as a closed subscheme of a Grassmanian.
	
	We define the \emph{wonderful compactification} $X$ of $G$ to be the closure of the $G \times G$-orbit
	of the point of $\mathbb{L}$ corresponding to the diagonally embedded
	$\mathfrak{g} \subset \mathfrak{g} \oplus \mathfrak{g}$.
\end{definition}

Note that $X(k) \neq \emptyset$.
As proved in \cite[\S6.1]{DCP83} over $\overline{k}$, the wonderful compactification is a smooth projective bi-equivariant compactification of $G$. That the same holds over $k$ follows from Galois descent. We denote by $\A$ the set of boundary divisors of $X$, and for $\alpha \in \A$ the corresponding divisor by $D_\alpha$. We let $\overline{\A}$
be the set of boundary divisors over $\overline{k}$. The geometric boundary divisor $\overline{D} = \cup_{\overline{\alpha} \in \overline{\A}} D_{\overline{\alpha}} = \overline{X} \setminus \overline{G}$ is a simple normal crossings divisor over $\overline{k}$ (see \cite[\S3.1]{DCP83} or \cite[\S2.1]{Bri07}).

Now suppose that $G$ is quasi-split, let $B \subset G$ be a Borel subgroup
and  $T \subset B \subset G$ a maximal torus. Let $B_-$ be the  opposite Borel subgroup
(i.e.~$B_- \cap B = T)$ and consider the open orbit $B_- B \subset G$. Its complement in $X$ consists of $(B_-\times B)$-stable divisors. Those $(B_-\times B)$-stable divisors which are not $(G\times G)$-stable are called  {\it colours}. Let $X^\circ$ be the complement of the colours in $X$ and let $Z$ be the Zariski closure of $T$ in $X^\circ$.

\begin{lemma}
\label{lem:spherical}
	Assume that $G$ is quasi-split with Borel subgroup $B$ and let $T \subset B$ be a maximal torus.
	
	\begin{enumerate}
		\item The variety $Z$, as defined above, is a smooth toric variety with respect to $T$.
		\item For each $\alpha \in \A$, 
		the scheme theoretic intersection $E_\alpha = D_\alpha \cap Z$ is an integral divisor in $Z$.
		Moreover $D_\alpha$ and $Z$ intersect transversely at every smooth point of $D_\alpha$.
		\item Let $f$ be a simple root of $\overline{T}$ with respect to $\overline{B}$, viewed as a regular function on $\overline{T}$.
		Then $\mathrm{div} f = E_{\overline{\alpha}} \subset \overline{Z}$ for some $\overline{\alpha} \in \overline{\A}$. 
		\item These constructions give Galois equivariant bijections between the following:
		\begin{itemize}
			\item The simple roots of $\overline{G}$ with respect to $\overline{B}$ and $\overline{T}$.
			\item The boundary divisors of $\overline{G} \subset \overline{X}$.
			\item The boundary divisors of $\overline{T} \subset \overline{Z}$.
		\end{itemize}
	\end{enumerate}

\end{lemma}
\begin{proof}
	We first assume that $k$ is algebraically closed. In this case,
	we use some of the descriptions given in \cite{DCP83}
	(see also \cite[\S2.1]{Bri07}).
	As explained in the proof of \cite[Thm~3.1]{DCP83} (cf.~\cite[Lem.~2.2]{DCP83} and \cite[Prop.~2.3]{DCP83}),
	there is an isomorphism 
	\[
	U^- \times \mathbb{A}^{|\A|} \times U \cong X^{\circ},
	\]
	where $U \subset B$ is the unipotent radical of $B$, and $Z$ is identified with the linear
	subspace $\mathbb{A}^{|\A|}$ as a  toric variety (this in particular shows part (1)).
	Moreover, let $x_\alpha$ be the coordinates on $\mathbb{A}^{|\A|}$ for $\alpha \in \A$.
	Then for a boundary divisor $D_\alpha$, the intersection $D_\alpha \cap X^{\circ}$ 
	is identified with the hyperplane $x_\alpha = 0$ in $X^{\circ}$.
	As the affine subspace $\mathbb{A}^{|\A|} \subset X^{\circ}$ clearly intersects these
	transversely in a smooth divisor, this shows part (2).
	Part (3) follows from the explicit description of the $D_\alpha$ given in the proof of \cite[Lem.~2.7]{EJ08}.

	Part (4) easily follows from the above, and the case of a non-algebraically closed field
	follows from a simple Galois descent argument.
\end{proof}

Returning to the case of general adjoint $G$, we have the following.

\begin{proposition} \label{prop:Picard}
The following  hold.
\begin{enumerate}
\item We have the exact sequence:
\begin{equation} \label{eqn:Picard}
0 \rightarrow \oplus_{\alpha \in \A} \mathbb Z D_\alpha \rightarrow \Pic(X) \rightarrow \Pic(G)\rightarrow 0.
\end{equation}

\item
The closed cone of effective divisors $\Lambda_{\mathrm{eff}}(X)$ is generated by the boundary components of $X$, i.e.,
\[
\Lambda_{\mathrm{eff}}(X) = \oplus_{\alpha \in \A} \mathbb R_{\geq 0} D_\alpha.
\]
\item There exist $\kappa_\alpha  \geq 0$ such that the anticanonical divisor is given by
\begin{equation} \label{def:kappa}
	-K_X = \sum_{\alpha \in \A} (\kappa_\alpha + 1) D_\alpha.
\end{equation}

\end{enumerate}
\end{proposition}
\begin{proof}
We first assume that $k$ is algebraically closed.
Part (1) holds for any smooth projective compactification of a semisimple algebraic group (see \cite[\S9.0]{San81}).  The second part is a well-known property of the wonderful compactification (see \cite[Ex.~2.2.4., Ex.~2.3.5]{Bri07}).
For the last part, a general formula of Brion \cite[Thm.~4.2]{Bri07} states that 
\begin{equation} \label{eqn:-K_X}
-K_X = \sum_{\alpha \in \A} D_\alpha + \sum_{\text{colours }D} m_D D,
\end{equation}
for some $m_D > 0$. The result follows.

Now assume that $k$ is non-algebraically closed. For part (1), we apply Galois cohomology to the sequence \eqref{eqn:Picard} over $\overline{k}$ to obtain
$$0 \rightarrow \HH^0(k,\oplus_{\overline{\alpha} \in \overline{\A}} \mathbb Z D_{\overline{\alpha}})) \rightarrow \HH^0(k,\Pic(\overline{X})) \rightarrow \HH^0(k,\Pic(\overline{G})) \rightarrow \HH^1(k,\oplus_{\overline{\alpha} \in \overline{\A}} \mathbb Z D_{\overline{\alpha}})).$$
By Shapiro's lemma we have $\HH^0(k,\oplus_{\overline{\alpha} \in \overline{\A}} \mathbb Z D_{\overline{\alpha}}) = \oplus_{\alpha \in \A} \mathbb Z D_{\alpha}$ and $\HH^1(k, \oplus_{\overline{\alpha} \in \overline{\A}} \mathbb Z D_{\overline{\alpha}}) = \HH^1(k,\oplus_{\alpha \in \A} \mathbb Z D_{\alpha}) = 0$. Moreover, as $G$ and $X$ both have no non-constant invertible functions and $\emptyset \neq G(k) \subset X(k)$, we have $\HH^0(k,\Pic(\overline{X})) = \Pic(X)$ and $\HH^0(k,\Pic(\overline{G})) = \Pic(G)$ (see e.g.~\cite[Lem.~6.3(iii)]{San81}). Part (1) now easily follows. Part (2) similarly follows by taking Galois invariants. Part (3) follows immediately from \eqref{eqn:-K_X}.
\end{proof}

\section{Zero-loci of Brauer group elements} \label{sec:zero-loci} \label{sec:main}
In this section we prove our main result (Theorem \ref{thm:arbitrarybiglinebundles}),
which is a generalisation of Theorem \ref{thm:semi-simple} to height functions
attached to arbitrary big line bundles.

\subsection{Set-up}  \label{sec:set-up}
Throughout $G$ is an adjoint semi-simple algebraic group over a number field $k$. 
We choose Haar measures $\mathrm{d}g_v$ on each $G(k_v)$ such that 
$G(\OO_v)$ has measure $1$ for all but finitely many $v$. The product of these measures converges
to a well-defined measure $\mathrm{d} g$ on $G(\Adele_k)$.
We choose our measures so that $\vol(G(\Adele_k)/G(k)) = 1$ for the induced quotient measure.
We let $G'$ denote the quasi-split inner form of $G$.

Let $\br \subset \Br_1 G$ be a finite subgroup such that $G(k)_\br \neq \emptyset$. By changing the group law of $G$ if necessary, we may take any element of $G(k)_\br$ to be the identity element. We may therefore assume that $e \in G(k)_\br$, i.e.~that $\br \subset \Br_e G$ (see \eqref{def:Br_e}).
We denote by $\R$ the group of automorphic characters of $G$
attached to $\br$ via Theorem \ref{thm:Br_e(G)}.  For a place $v$ of $k$,
we denote by $\thorn_v: G(k_v) \to \{0,1\}$ the indicator function of $G(k_v)_\br = \cap_{\rho \in \R} \ker \rho_v$. We let
$\thorn = \prod_v \thorn_v: G(\Adele_k) \to \{0,1\}$ be the product of the local indicator functions with zero locus $G(\Adele_k)_\br$. Note that the Hasse principle
for $\Br k$ (see \eqref{seq:CFT}) implies that $G(k)_\br = G(k) \cap G(\Adele_k)_\br$.

We let $G \subset X$ be the wonderful compactification of $G$ with set of boundary
divisors $\A$ (see \S \ref{sec:wonderful}). For $\alpha \in \A$ 
we let the $\kappa_\alpha$ be as in Proposition \ref{prop:Picard} and let $k_\alpha$
be the field of constants of $D_\alpha$ (i.e.~the algebraic closure of $k$ inside the function field
of $D_\alpha$).

We have the following description for a maximal torus of $G'$ via the Weil restriction.

\begin{lemma} \label{lem:WR}
	Let $B' \subset G'$	be a Borel subgroup and $T' \subset B'$ a maximal torus. Then
	\begin{equation} \label{eqn:Weil_restriction}
		T' \cong \prod_{\alpha \in \A} \Res_{k_\alpha/k} \Gm.
	\end{equation}
\end{lemma}
\begin{proof}
	As $G'$ is quasi-split and adjoint, the simple roots form a basis of the character group $X^*(\overline{T}')$. 
	However, the simple
	roots are permuted by the Galois action in the same way as the divisors $D_\alpha$ (see Lemma \ref{lem:spherical});
	the result easily follows.
\end{proof}

\subsection{Heights} \label{sec:heights}
We assume that the reader is familiar with the 
theory of heights attached to adelic metrics on line bundles, as can be found, for example, in \cite[\S2.1, 2.2]{CLT10}.
Denote by $\mathcal{T} = \CC^\A$ the complex vector space with basis given by the elements of $\A$.
We define an adelic complexified height pairing
$$H(\sbf, g): \quad \CC^{\A} \times G(\Adele_k) \to \CC$$
as well as local height pairings, as follows.
Choose smooth adelic metrics $|| \cdot ||_\alpha$ on the line bundles $\OO_X(D_\alpha)$ and let $d_\alpha$
be a choice of global section of $\OO_X(D_\alpha)$ which vanishes on $D_\alpha$. 
We define
\begin{equation} \label{def:height}
H_v(\sbf, g_v) = \prod_{\alpha \in \A}||d_\alpha(g_v)||_v^{-s_\alpha}, \quad
H(\sbf,(g_v)_{v})= \prod_{v}H_v(\sbf, g_v).
\end{equation}
Note that $H_v$ depends on the choice of the $d_\alpha$, but
the adelic height $H$ is independent of the $d_\alpha$ by the product formula.
From Proposition \ref{prop:Picard}, the natural map $\ZZ^\A \to \Pic(X)$ has finite cokernel,
hence induces an isomorphism $\CC^\A \cong \Pic(X)_\CC:=\Pic(X)\otimes_\ZZ\CC$. This therefore yields an adelic
complexified height pairing $\Pic(X)_\CC \times G(\Adele_k) \to \CC$.

For example, define $\underline{\kappa} = (\kappa_\alpha + 1)_{\alpha \in \A}$ (see Proposition \ref{prop:Picard}).
Then $H(\underline{\kappa},\cdot)$
is a choice of anticanonical height function on $X$.
For $c \in \RR$, we let 
\begin{equation} \label{def:T_c}
	\mathcal{T}_c = \{ (s_\alpha) \in \CC^\A: \re s_\alpha > \kappa_\alpha + 1 + c, \forall \alpha \in \A\}.
\end{equation}
We use this notation since, as we shall prove, the height zeta
function is absolutely convergent on $\mathcal{T}_0$ and has singularities
along the hyperplanes  $\re s_\alpha = \kappa_\alpha + 1$. 

\begin{remark} \label{rem:heights}
Note that given a line bundle $L$, every smooth adelic metric on $L$ arises via \eqref{def:height}. Explicitly, choose a basis of $\Pic(X)_\QQ$ which contains $[L]$ together with adelic
metrics on each basis element. Expressing each $D_\alpha$ in terms of this new basis yields, via the above construction,
the original metric on $L$.
\end{remark}


We next consider the bi-invariance of the height and $\thorn$.
\begin{lemma} \label{lem:K_v} 
	For any nonarchimedean place $v$ of $k$, there exists a compact open subgroup $\K_v \subset G(k_v)$
	such that both $H_v$ and $\thorn_v$
	are bi-$\K_v$-invariant.
	Moreover, one may take $\K_v = G(\OO_v)$ for all but finitely many $v$.
\end{lemma}
\begin{proof}
	This property of the height is shown in \cite[Prop.~6.3]{JAMS}, 
	whilst the analogue for $\thorn_v$ is Lemma \ref{lem:K-invariant}.
\end{proof}

For every non-archimedean place $v$ of $k$, we now fix the choice of such a $\K_v$ satisfying Lemma \ref{lem:K_v}. 
and let $\K_0 = \prod_{v \nmid \infty} \K_v$. 
\subsection{Review of automorphic forms}  

Here we review some notation and results from the theory of automorphic forms, following Arthur's exposition \cite[\S1-3]{Arthur78}.

\subsubsection{Notation and basic definitions}
Fix a minimal parabolic subgroup $B$ of $G$. We denote standard parabolic subgroups of $G$ by $P$. For such a parabolic subgroup, let 
\begin{itemize}
\item $M_P$, the Levi factor, 
\item $N_P$, the unipotent radical, 
\item $A_P$, the split component of the center of $M_P$, 
\item $\mathfrak X^*(M_P)_k$, the group of characters of $M_P$ defined over $k$, 
\item $\mathfrak a_P = \Hom(\mathfrak X^*(M_P)_k, \RR)$, $\mathfrak a_P^* = \mathfrak X^*(M_P)_k \otimes \RR$, 
\item $(\mathfrak a_P)_\CC = \mathfrak a_P \otimes_\RR \CC$, $(\mathfrak a_P^*)_\CC = \mathfrak a_P^* \otimes_\RR \CC$,
\item $\Delta_P$, the set of the simple roots of $(P, A_P)$, \item
$\mathfrak a_P^+ = \{ \mathsf{H} \in \mathfrak a_P; \alpha(\mathsf{H}) > 0 \text{ for all } \alpha
\in \Delta_P\}$, 
\item $(\mathfrak a_P^*)^+ = \{ \Lambda \in \mathfrak a_P^*; \Lambda(\check{\alpha}) > 0
\text{ for all } \alpha \in \Delta_P\},$ where in this expression $\check\alpha$ is the coroot associated with $\alpha$, 
\item $n(A_P)$, be the number of chambers in $\mathfrak a_P$.
\end{itemize}

If there is no confusion, we drop the subscript $P$. 
Also, recall our choices of the groups $\K_v$ following Lemma \ref{lem:K_v}.  For each archimedean place $v$ of $k$, we choose a maximal compact subgroup $\K_v \subset G(k_v)$ such that 
$$
G(k_v) = \K_v A_B(k_v) \K_v. 
$$
We set $\K = \prod_v \K_v$, $\K_0 = \prod_{v \text{ non-arch.}} \K_v$, $\K_\infty = \prod_{v \text{ arch.}} \K_v$.

\

Let ${\mathsf W}$ be the restricted Weyl group of $(G, A_B)$.  The group ${\mathsf W}$ naturally acts on $\mathfrak a_B$ and $\mathfrak a_B^*$. For any $\mathsf s\in {\mathsf W}$, fix a representative $w_{\mathsf s}$ in the
intersection of $G(k)$ with the normalizer of $A_B$. For parabolic subgroups $P_1, P_2$, let ${\mathsf W}(\mathfrak a_{P_1}, \mathfrak a_{P_2})$ be the collection of distinct isomorphisms $\mathfrak a_{P_1} \to \mathfrak a_{P_2}$ obtained by
restricting elements of ${\mathsf W}$ to $\mathfrak a_{P_1}$. A pair of parabolic subgroups $P_1, P_2$ are called {\em associated} if ${\mathsf W}(\mathfrak a_{P_1}, \mathfrak a_{P_2})$ is not empty.  This is an equivalence relation. 

\

Let $P$ be a parabolic subgroup with Levi factor $M$ and unipotent radical $N$. For $m = (m_v)_v \in M(\Adele_k)$, we define an element $\mathsf H_M(m) \in \mathfrak a_P$
by
\begin{equation}
e^{\langle \mathsf H_M(m) , \chi \rangle} = |\chi(m)| = \prod_v
|\chi(m_v)|_v
\end{equation}
for all $\chi \in \mathfrak X^*(M)_k$. This is a homomorphism
\begin{equation}
M(\Adele_k) \longrightarrow \mathfrak a_P.
\end{equation}
We let $M(\Adele_k)^1$ be the kernel. By Iwasawa decomposition, any $x \in G(\Adele_k)$ can be written as
$nma\kappa$ with $n \in N(\Adele_k), m \in M(\Adele_k)^1, a \in \mathsf \prod_{v \text{ arch.}}A_P(k_v)^0, \kappa \in
\K$. Set $\mathsf H_P(x) := \mathsf H_M(a) \in \mathfrak a_P$. There is a vector $\rho_P \in(\mathfrak a_P^*)^+$ such that
\begin{equation}
\delta_P(p) = |\det (Ad\, p|_{\mathfrak n_P(\Adele_k)})| = e^{2\rho_P(\mathsf H_P(p))}
\end{equation}
for all $p\in P(\Adele_k)$. 

\

We now explain the normalization of measures. Recall that we chose a Haar measure on $G(\Adele_k)$ in \S \ref{sec:set-up}. For any subgroup $V$ of $N_B$, the unipotent radical of the minimal parabolic subgroup $B$, we choose a Haar measure on $V(\Adele_k)$ so that $V(k)\backslash V(\Adele_k)$ has volume  one. We equip $\K$ with the Haar measure with total volume one. Fix Haar measures on each of the vector spaces $\mathfrak a_P$ so that the lattice $\Hom(\mathfrak X^*(M_P)_k, \ZZ) \subset \mathfrak a_P$ has covolume $1$, and use the dual measures on $\mathfrak a_P^*$.   For any standard parabolic $P$, there is a unique measure on $M_P(\Adele_k)$ such that 
$$
\int_{G(\Adele_k)} f(x) \, \mathrm{d}x = \int_{N_P(\Adele_k)}\int_{M_P(\Adele_k)}\int_\K f(nm\kappa) e^{-2 \rho_P(\mathsf H_P(m))} \, \mathrm{d}n \, \mathrm{d}m \, \mathrm{d}\kappa. 
$$
The map $\mathsf H_M: M_P(\Adele_k) \to \mathfrak a_P$ determines an isomorphism $M_P(\Adele_k)/M_P(\Adele_k)^1 \to \mathfrak a_P$.  Since we already have a Haar measure on the group $M_P(\Adele_k)$, there is a unique Haar measure on the group $M_P(\Adele_k)^1$ such that the quotient measure on $M_P(\Adele_k)/M_P(\Adele_k)^1$ is the one obtained by transport of structures from $\mathfrak a_P$. 

\subsubsection{Eisenstein series} 
We now recall the definition of Eisenstein series in general. We will do this in several stages. 

\

1. Let $M$ be the Levi factor of some standard parabolic subgroup $P$ of $G$.  
Let $L^2_{cusp}(M(k) \backslash M(\Adele_k)^1)$ be the space of
functions $\phi$ in $L^2(M(k) \backslash M(\Adele_k)^1)$ such that
for any parabolic $Q \subsetneqq P$ we have
\begin{equation}
\int_{N_Q(k) \cap M(k) \backslash N_Q(\Adele_k) \cap M(\Adele_k)}
\phi(nm) \, \mathrm{d}n = 0
\end{equation}
for almost all $m$. It is known (\cite[Lemma 5 and its corollary]{Arthur79}) that 
\begin{equation}
L^2_{cusp}(M(k) \backslash M(\Adele_k)^1) = \bigoplus_\varrho V_\varrho
\end{equation}
where $\varrho$ ranges over all irreducible unitary representations of $M(\Adele_k)^1$, and $V_\varrho$ is the $\varrho$-isotypic component of $\varrho$ consisting
of finitely many copies of $\varrho$ (possibly zero).

\

2. We define an equivalence relation on the set of pairs $(M, \varrho)$ with $M$ a Levi factor of some standard parabolic subgroup of $G$ and $\varrho$ an irreducible unitary representation of $M(\Adele_k)^1$ occurring in $L^2_{cusp}(M(k) \backslash M(\Adele_k)^1)$. For two such pairs $(M, \varrho)$ and $(M', \varrho')$, set  $(M, \varrho) \sim (M', \varrho')$  if there are parabolic subgroups $P, P'$ with Levi factors $M, M'$, respectively, such that there is an $\mathsf s
\in {\mathsf W}(\mathfrak a_P, \mathfrak a_{P'})$ with the property that the representation
\begin{equation}
(\mathsf s\varrho)(m') = \varrho(w_{\mathsf s}^{-1} m' w_{\mathsf s}) \,\,\,\,\,(m' \in M'(\Adele_k)^1)
\end{equation}
is unitarily equivalent to $\varrho'$. Let $\mathfrak X$ be the set of
equivalence classes. For any $\chi \in \mathfrak X$ we have a
class $\mathcal{P}_\chi$ of associated parabolic subgroups consisting of those parabolic subgroups $P$ with a Levi subgroup $M$ and a representation $\rho$ such that $(M, \rho) \in \chi$. 

\

3. If $M$ is the Levi factor of some parabolic subgroup and $\chi \in \mathfrak X$, set
\begin{equation}
L^2_{cusp}(M(k) \backslash M(\Adele_k)^1)_\chi = \bigoplus_{(\varrho: (M,
\varrho) \in \chi)}V_\varrho.
\end{equation}
This is a closed subspace of $L^2_{cusp}(M(k) \backslash
M(\Adele_k)^1)$, and zero if $P \notin \mathcal{P}_\chi$ for every parabolic subgroup $P$ that has $M$ as a Levi factor. Then we have 
\begin{equation}
L^2_{cusp}(M(k) \backslash M(\Adele_k)^1)= \bigoplus_{\chi \in \mathfrak X}L^2_{cusp}(M(k) \backslash M(\Adele_k)^1)_\chi.
\end{equation}

4. Let $M$ be the Levi factor of some parabolic subgroup $P$, and fix an equivalence class $\chi \in \mathfrak X$. Suppose there is a $P_1 \in \mathcal{P}_\chi$ such that $P_1 \subset P$. Let $\psi$ be a smooth
function on $N_{P_1}(\Adele_k)M_{P_1}(k) \backslash G(\Adele_k)$ such that
\begin{equation}
\Psi_a(m, \kappa) = \psi(am\kappa), \quad \kappa \in \K, m \in
M_{P_1}(k)\backslash M_{P_1}(\Adele_k), a \in A_{P_1}(k)\backslash A_{P_1}(\Adele_k)
\end{equation}
vanishes outside a compact subset of $A_{P_1}(k)\backslash
A_{P_1}(\Adele_k)$, transforms under $\K_\infty$ according to an irreducible
representation, and as a function of $m$ belongs to the space
$L^2_{cusp}(M_{P_1}(k) \backslash M_{P_1}(\Adele_k)^1)$. Then the function
\begin{equation}
\hat{\psi}^M(m) = \sum_{\gamma \in P_1(k) \cap M(k) \backslash
M(k)} \hspace{-20pt} \psi(\gamma m), \quad \quad m \in M(k) \backslash
M(\Adele_k)^1
\end{equation}
is square-integrable on $M(k) \backslash M(\Adele_k)^1$. Define
$L^2(M(k) \backslash M(\Adele_k)^1)_\chi$ to be the span of all such
$\hat{\psi}^M$. If no such $P_1$ exist, the latter space is defined to be the zero space. By \cite[Lemma 2]{Langlands65} for $\QQ$, and 
	\cite[II.2.4.Proposition]{MW} for the general case, we have
\begin{equation}
L^2(M(k) \backslash M(\Adele_k)^1) = \bigoplus_\chi L^2(M(k)
\backslash M(\Adele_k)^1)_\chi.
\end{equation}

\

5. If $M$ is the Levi factor of some standard parabolic subgroup $P$, let $\Pi(M)$ be the set of equivalence classes of
irreducible unitary representations of $M(\Adele_k)$. For $\zeta \in
\mathfrak (a_P^*)_\CC$ and $\pi \in \Pi(M)$ let $\pi_\zeta$ be the product of
$\pi$ with the quasi-character
\begin{equation}\label{eq:pizeta}
x \mapsto e^{\zeta \mathsf H_P(x)} \,\,\,\,\,\,(x \in G(\Adele_k)).
\end{equation}
If $\zeta \in i\mathfrak a_P^*$ then $\pi_\zeta$ is again unitary. This gives
$\Pi(M)$ the structure of a differentiable manifold, with infinitely many connected components, which carries an action of
$i\mathfrak a_P^*$. Since the connected components are identified with $i \mathfrak a_P^*$, we also obtain a measure $\mathrm{d}\pi$ on
$\Pi(M)$ .

\

6. Let $P$ be a standard parabolic subgroup. For $\pi \in \Pi(M_P)$ we let $\mathcal{V}^0_P(\pi)$ be the space of
smooth functions
\begin{equation}
\phi: N_P(\Adele_k) M_P(k) \backslash G(\Adele_k) \to \CC
\end{equation}
satisfying
\begin{enumerate}
\item $\phi$ is right $\K$-finite;
\item for every $x \in G(\Adele_k)$ the function $$m \mapsto
\phi(mx),\quad m \in M_P(\Adele_k)$$ is a matrix coefficient of $\pi$;
\item $\| \phi \|^2 = \int_\K \int_{M_P(k)\backslash M_P(\Adele_k)^1}
|\phi(mk)|^2 \, dm \, d\kappa < +\infty$.
\end{enumerate}
Let $\mathcal{V}_P(\pi)$ be the completion of $\mathcal{V}^0_P(\pi)$ with respect to $\|\cdot \|$. For $\phi \in
\mathcal V_P(\pi)$ and $\zeta \in (\mathfrak a_P^*)_\CC$ set
\begin{equation}
\phi_\zeta(x)  = \phi(x) e^{\zeta(\mathsf H_P(x))}, \quad x \in G(\Adele_k)
\end{equation}
and
\begin{equation}
(I_P(\pi_\zeta, y)\phi_\zeta)(x) = \phi_\zeta(xy)
\delta_P(xy)^{\frac{1}{2}} \delta_P(x)^{-\frac{1}{2}}.
\end{equation}
Then $I_P(\pi_\zeta)$ is a unitary representation if $\zeta \in i\mathfrak a_P^*$. Note that $I_P(\pi_\zeta)$ is the global parabolically induced representation of $\pi_\zeta$ from $P$ to $G$. 

\

7. Given $\chi \in \mathfrak X$, let $\mathcal V_P(\pi)_\chi$ be the closed subspace of
$\mathcal V_P(\pi)$ consisting of those $\phi$ such that for all $x$ the
function $m \mapsto \phi(mx)$ belongs to $L^2(M_P(k) \backslash
M_P(\Adele_k)^1)_\chi$. Then
\begin{equation}\label{eq:Hp}
\mathcal V_P(\pi) = \bigoplus_{\chi} \mathcal V_P(\pi)_\chi.
\end{equation}
Let $\mathcal V_P(\pi)_{\chi, \K_0}$ be the subspace of functions in
$\mathcal V_P(\pi)_\chi$ which are invariant under $\K_0$. Let $W$ be an equivalence class of irreducible representations of $\K_\infty$, and define
$\mathcal V_P(\pi)_{\chi, \K_0, W}$ to be the space of those functions in $\mathcal V_P(\pi)_{\chi,
\K_0}$ which transform under $\K_\infty$ according to $W$. By \cite[\S 7]{Langlands76} each of the spaces $\mathcal V_P(\pi)_{\chi, \K_0, W}$ is finite-dimensional. We fix an orthonormal basis
$\mathcal B_P(\pi)_\chi$ for $\mathcal V_P(\pi)_\chi$, for each $\pi$ and each $\chi$, 
  such that for all $\zeta \in i \mathfrak a^*$ we have
\begin{equation}\label{eq:Bpi}
\mathcal B(\pi_\zeta)_\chi = \{ \phi_\zeta : \phi \in \mathcal B_P(\pi)_\chi \}
\end{equation}
and such that every $\phi \in\mathcal B_P(\pi)_{\chi}$ belongs to one of the
spaces $\mathcal V_P(\pi)_{\chi, \K_0, W}$.

\

8. Let $M$ be the Levi factor of some standard parabolic subgroup $P$. Suppose $\pi \in \Pi(M)$, $ \phi \in \mathcal B_P(\pi)_\chi $.  We would like to define an Eisenstein series $E(x, \phi)$. If we formally set
\begin{equation}
E(x, \phi) = \sum_{\gamma \in P(k) \backslash G(k)}
\phi(\gamma x) \delta_P(\gamma x)^{\frac{1}{2}}, 
\end{equation}
then the series does not converge. For this reason we have to use analytic continuation. For  $\phi \in \mathcal V_P^0(\pi)$, and  $\zeta \in
(\mathfrak a_P^*)_\CC$ with $\re(\zeta) \in \rho_P + (\mathfrak a_P^*)^+$ we define  
\begin{equation}
E(x, \phi_\zeta) = \sum_{\gamma \in P(k) \backslash G(k)}
\phi_\zeta(\gamma x) \delta_P(\gamma x)^{\frac{1}{2}}.
\end{equation}
For such $\zeta$  the series is absolutely convergent. The function $E(x, \phi_\zeta)$ 
can be analytically continued to a meromorphic function on $(\mathfrak a_P^*)_\CC$. For $\zeta \in i\mathfrak a_P^*$, the analytically continued $E(x, \phi_\zeta)$ is a smooth function of $x$. The value of this analytically continued function at $\zeta=0$ is what we denote by $E(x, \phi)$.  The map that sends a form $\phi_\zeta$ to the Eisenstein series $E(\cdot , \phi_\zeta)$, when defined, gives an intertwining map from $I_P(\pi_\zeta)$ to $L^2(G(k)\backslash G(\Adele_k))$.

\subsection{Spectral expansion} 

Let $F$ be a function on $G(k) \backslash G(\Adele_k)$. By
the {\em spectral expansion} of $F$ we mean
\begin{equation}\label{spectral-identity-general}
S(F, x) := \sum_{\chi \in \mathfrak X}\sum_{P}  n(A_P)^{-1}  \hspace{-10pt}
 \int\limits_{\Pi(M_P)}\int\limits_{G(k) \backslash G(\Adele_k)}\hspace{-9pt}
 \left(\sum_{\phi \in \mathcal B_P(\pi)_{\chi}} \hspace{-10pt} E(x, \phi) \overline{E(y, \phi)}\right)F(y) \, \mathrm{d}y\, \mathrm{d}\pi.
\end{equation}

For $P = G$, the space $\Pi(G)$ is a discrete space parametrising  automorphic characters, irreducible cuspidal representations and other residual representations; moreover our normalisations imply that the measure $\mathrm{d} \pi$ is simply the counting measure in this case. 
 In general, if $F$ is an arbitrary $L^2$ function, there is no reason that $S(F, x)$ should be equal to $F$; in fact, in a lot of cases, $S(F, x)$ does not even converge. One of the main technical theorems of \cite{JAMS} is concerned with describing a set of conditions for a function $F$ so that $S(F, x) = F(x)$, with $S(F, x)$ interpreted appropriately, for {\em all} $x$.

\

We now describe the relevant result from \cite{JAMS}. Let $f$ be a smooth function on $G(\Adele_k)$, and suppose that $f$ is
right invariant under a compact-open subgroup $\K_0$ of $G(\Adele_f)$.
Define a function on $G(k) \backslash G(\Adele_k)$ by
$$
F(g) := \sum_{\gamma \in G(k)} f(\gamma g).
$$
Suppose that $f$ is such that the function $F$ is convergent for
all $g$, and defines a smooth and bounded function.

\begin{theorem}[Lemma 3.1 of \cite{JAMS}]\label{lem:spectral-expansion-general}
Let $H = H(\underline{\kappa},\cdot)$ be the adelic anticanonical height function on $G(\Adele_k)$ from \S\ref{sec:heights}.
There is a differential operator $\Delta$, chosen as in \cite[Lemma 4.1]{Arthur78} , and natural numbers $u, v$, such that if 
\begin{equation}\label{integral1-spectral}
\int_{G(\Adele_k)} |f(g)|\cdot  H(g)^u \, \mathrm{d}g
\end{equation}
and
\begin{equation}\label{integral2-spectral}
\int_{G(\Adele_k)} |\Delta^v f(g)|\cdot H(g)^u \, \mathrm{d}g
\end{equation}
converge, then 
\begin{equation}\label{integral3-spectral}
F(x) = S(F, x) \quad \quad \text{for all } x \in G(k)\backslash G(\Adele_k).
\end{equation}

\end{theorem}
\begin{remark}
Arthur uses the notation $Z$ for the operator $\Delta$.  Since Arthur works only over $\QQ$ we explain the construction for an arbitrary number field.  Let $G_\infty = \prod_{v \text{ arch.}} G(F_v)$ and $\K_\infty = \prod_{v \text{ arch.}} \K_v$. Let $\Omega_G$ and $\Omega_\K$ be the Casimir operators of the Lie groups $G_\infty$ and $\K_\infty$, respectively. We define $\Delta$, sometimes called the Laplacian, to be the operator $\Omega_G -2 \Omega_\K$. It is not hard to see that $\Delta$ is elliptic, and so all of its eigenvalues are non-negative real numbers. Much of the theory of automorphic forms is concerned with understanding the discrete part of the spectrum of $\Delta$. Once the discrete part is understood, the theory of Eisenstein series provides the tools to understand the continuous part of the spectrum. 
\end{remark}

We will also need the following proposition: 

\begin{proposition}[Proposition 3.5 of \cite{JAMS}] \label{important-proposition}
For $\phi \in \mathcal B_P(\pi)_{\chi}$ define $\Lambda(\phi)$
by $\Delta\cdot \phi = \Lambda(\phi)\cdot \phi$. Then there is
an $\ell>0$ such that 
\begin{equation}
\sum_{\chi \in \mathfrak X}\sum_{P} n(\mathsf A_P)^{-1} \int_{\Pi(M_P)}\left( \sum_{\phi
\in \mathcal B_P(\pi)_{\chi}} \Lambda(\phi)^{-\ell}|E(e, \phi)|^2 \right)\, \mathrm{d}\pi
\end{equation}
is convergent.
The outermost summation is only over those classes $\chi$ for which there is a $\K_0$-fixed $\phi$ with $\Lambda(\phi) \ne 0$ and the inner most summation is over such $\phi$. Note $\Lambda(\phi) > 0$.  
\end{proposition}

\subsection{Bounds for matrix coefficients}\label{bounds:matrix}

One of the key ingredients of the proof of the main theorem of \cite{JAMS} was certain bounds for matrix coefficients of admissible representations of the group of rational points of a semi-simple group of adjoint type over a local field. If the local rank of the group is $1$, these bounds follow from estimates towards the Ramanujan conjecture for automorphic representations; whereas, for rank larger than $1$, the estimates are purely local, \cite{Oh}. The way these estimates make an appearance in \cite{JAMS} is via the connection of integrals of the form 
$$
\int_{\K}\int_\K \phi( \kappa g \kappa') \, \mathrm{d} \kappa \, \mathrm{d} \kappa', 
$$
for $\phi$ an automorphic form on the group $G$, to products of local spherical functions. For simplicity this connection was made explicit only for a special case in \cite[Corollary 4.1]{JAMS}. In the present work, however, we cannot make the simplifying assumption made in \cite{JAMS}. Here we explain the necessary modifications to the argument of \cite[\S 4.1]{JAMS}. 

\

Let $\mathcal H(G(\Adele_k))$ be the global Hecke algebra of the group $G(\Adele_k)$, and let $\xi \in \mathcal H(G(\Adele_k))$ be a non-trivial idempotent. (For basic material on Hecke algebras see \cite[Ch. III]{Jacquet-Langlands} or \cite{Flath}). We assume that $\xi= \otimes_v' \xi_v$ where for each $v$ the local idempotent $\xi_v$ is chosen as follows: 
\begin{itemize} 
\item For $v$ non-archimedean, $\xi_v = \vol(\K_v)^{-1} ch_{\K_v}$ with $\K_v$ the compact open subgroup of Lemma \ref{lem:K_v}, where here $ch_{\K_v}$ is the characteristic function of the set $\K_v$.  
\item For each archimedean place $v$, let $W_v$ be a finite dimensional smooth representation of $\K_v$, and let $\tr_{W_v}$ be its character extended as a function on $G(k_v)$ by defining it to be zero outside $\K_v$. We let $\xi_v = (\dim W_v) \cdot \tr_{W_v}$.  Let $W = \otimes_{v \text{ arch.}} W_v$.  
\end{itemize}
Let $\phi \in \mathcal{V}^0_P(\pi)$ be a vector of norm $1$, and suppose that $\phi * \xi = \phi$, with $*$ the convolution action of $\mathcal H(G(\Adele_k))$ on $\mathcal V^0_P(\pi)$.   

\

We wish to find a bound for 
$$
M_\xi(g, \phi) : = (\xi* E)(g, \phi).
$$
 We will explain the details for the case where the semi-simple rank of $G$ is larger than or equal to $2$ over every localization of the ground field. The extension to the general case works along the lines of \cite{JAMS}, especially Theorem 4.5.

We need a piece of notation from \cite{Oh}. Suppose $H$ is a reductive group defined over a local field $F$, $A$ a maximal $F$-split torus in $H$, and $K$ a good maximal compact subgroup such that we have the Cartan decomposition $H(F) = K A^+ \Omega K$, with $A^+$ a closed positive Weyl chamber and $\Omega$ a finite subset of $H(F)$. Let $\Phi$ be the set of 
non-multipliable roots of $A$ and $\Phi^+$ the set of positive roots. We call a set $\mathcal S \subset \Phi^+$ a {\em strongly orthogonal system} if for every two distinct $\alpha, \beta \in \mathcal S$, neither $\alpha \pm \beta$ is an element of $\Phi$.  Let $\Xi_{\PGL_2(F)}$ be the Harish-Chandra function of $\PGL_2(F)$ defined as follows:
\begin{itemize}
\item If $F = \RR$ and $x \geq 1$, 
$$
\Xi_{\PGL_2(\RR)}(x) = \frac{2}{\pi \sqrt{x}} \int_0^{\pi/2} \left( \frac{\cos^2 t}{x^2} + \sin^2 t\right)^{-1/2} \, dt. 
$$
For other values of $x$,  $\Xi_{\PGL_2(\RR)} ( x) = \Xi_{\PGL_2(\RR)}(\max\{|x|, |x|^{-1}\})$; 
\item If $F = \CC$ and $x \geq 1$, 
$$
\Xi_{\PGL_2(\CC)}(x) = \frac{2}{\pi x} \int_0^{\pi/2} \left( \frac{\cos^2 t}{x^2} + \sin^2 t\right)^{-1/2} \sin(2t) \, dt. 
$$
Again, for an arbitrary $x$, $\Xi_{\PGL_2(\CC)} ( x) = \Xi_{\PGL_2(\CC)}(\max\{|x|, |x|^{-1}\})$; 
\item If $F$ is non-archimedean with uniformizer $\varpi$ and $|\varpi| = q^{-1}$ then for $n \in \NN$, 
$$
\Xi_{\PGL_2(F)}(\varpi^n) = \frac{1}{ q^{n/2}} \left(\frac{n (q-1) + (q+1)}{q+1} \right). 
$$
Furthermore, $\Xi_{\PGL_2(F)} (x^{-1}) = \Xi_{\PGL_2(F)} (x)$ for each $x$, and for every unit $\epsilon$,  $\Xi_{\PGL_2(F)} (\epsilon x) = \Xi_{\PGL_2(F)} (x)$.
\end{itemize}
By \cite[\S 2.2]{Oh}, $\Xi_{\PGL_2(F)}$ takes values in $(0, 1]$, and for any $\varepsilon > 0$, there are constants $c_1, c_2(\varepsilon)$ such that for all $x \in F^\times$, 
\begin{equation}\label{Xi}
c_1 \max \{|x|, |x|^{-1} \}^{-1/2} \leq \Xi_{\PGL_2(F)}(x)  \leq c_2(\varepsilon) \max\{|x|, |x|^{-1} \}^{-1/2+ \varepsilon}. 
\end{equation}

 For $a \in A^+$, $k_1, k_2 \in K$, and $\omega \in \Omega$, we set 
$$
\xi_S(k_1 a \omega k_2) = \prod_{\alpha \in \mathcal S} \Xi_{\PGL_2(F)} (\alpha(a)). 
$$

\begin{lemma}\label{lem:4.6} Let the notations be as above. For each place $v$, pick a strongly orthogonal system in $G(k_v)$.  There is a constant $C_\xi$, depending only on the idempotent $\xi$, such that 
$$
|M_\xi( g, \phi)| \leq C_\xi  \sqrt{\dim \mathcal V_P(\pi)_{\chi, \K_0, W}} 
\max_{\phi \in \mathcal B_P(\pi)_\chi \cap \mathcal V_P(\pi)_{\chi, \K_0, W}}
 \{ |E(e, \phi)|\} \cdot \prod_v \xi_{{\mathcal S}_v} (g_v). 
$$
\end{lemma}
\begin{proof}
We set 
$$
\lambda_\xi(\phi) = (\xi* E)(e, \phi).  
$$
Then $\lambda_\xi$ is a smooth functional on the space $\mathcal{V}^0_P(\pi)$. Since $\mathcal{V}^0_P(\pi) \simeq \otimes_v' I_P(\pi_{v})$, we have
$$
\lambda_\xi \in \widetilde{\mathcal{V}^0_P(\pi)} := \otimes_v' \widetilde{I_P(\pi_{v})}, 
$$
with $\widetilde{I_P(\pi_{v})}$ the contragradient of the local representation $I_P(\pi_{v})$. 

Without loss of generality we may assume that $\lambda_\xi = \otimes_v \lambda_{v, \xi}$ where for almost all $v$, $\lambda_{v, \xi}$ is the $\K_v$-fixed vector in $\widetilde{I_P(\pi_{v})}$. Next, if via the identification $\mathcal{V}^0_P(\pi) \simeq \otimes_v' I_P(\pi_{v})$, $\phi$ is a pure tensor $\otimes_v \phi_{v}$, then, if $g = (g_v)_v$, we have  
$$
M_\xi(g, \phi) = \prod_v \langle \pi_{v}(g_v) \phi_{v}, \lambda_{v, \xi}\rangle_v 
$$
with $\langle ., . \rangle_v$ the pairing $I_P(\pi_{v}) \times \widetilde{I_P(\pi_{v})} \to \CC$.  Now we invoke the main theorem of \cite{Oh}, or \cite[\S 4.2]{JAMS} for a summary, to obtain the bound 
$$
\langle \pi_{v}(g_v) \phi_{v}, \lambda_{v, \xi}\rangle_v \leq (\dim \K_v \phi_{v})^{1/2} (\dim \K_v \lambda_{x, v})^{1/2} || \phi_{v} || \cdot ||\lambda_{x, v} || \cdot \xi_{\mathcal S} (g_v) 
$$
for any strongly orthogonal set $\mathcal S$. We note that for almost all $v$, $$(\dim \K_v \phi_{v})^{1/2} (\dim \K_v \lambda_{x, v})^{1/2}=1,$$ so we may safely multiply all of these inequalities to obtain 
\begin{equation}\label{M-bound}
|M_\xi(g, \phi)| \leq C'_\xi (\dim \K\phi)^{1/2} ||\lambda_\xi||\cdot ||\phi|| \cdot \prod_v \xi_{\mathcal S} (g_v), 
\end{equation}
with $C'_\xi$ a constant depending only on $\xi$. 
\

Next we estimate $||\lambda_\xi||$. The important point to keep in mind is that $||\lambda_\xi||$ depends on the space on which $\lambda_\xi$ is acting. Fix $\chi \in \mathfrak X$ and $\pi$ and consider the restriction of $\lambda_\xi$ to $\mathcal V_P(\pi)_\chi$. We  have $\lambda_\xi(\phi) = \xi * E(e, \phi) = E(e, \xi * \phi)$. Since convolution with $\xi$ is a projection $\mathcal V_P(\pi)_\chi \to \mathcal V_P(\pi)_{\chi, \K_0, W}$, we conclude that 
\begin{equation}\label{lambda-bound}
||\lambda_\xi || \leq 
\sqrt{\dim \mathcal V_P(\pi)_{\chi, \K_0, W}} 
\max_{\phi \in \mathcal B_P(\pi)_\chi \cap \mathcal V_P(\pi)_{\chi, \K_0, W}}
 \{ |E(e, \phi)|\}. 
\end{equation}

Here we have used the following statement: If $V$ is a finite dimensional complex Hilbert space and $\lambda:V \to \CC$ is a linear functional, then for any orthonormal basis $\mathcal B$ of $V$ the linear map norm of $\lambda$ is bounded by $\sqrt{\dim V} \max_{v \in \mathcal B} \{ |\lambda(v)|\}$.  For convenience, we include a proof. Let $w \in V$. Then $w = \sum_{c \in \mathcal B} c_v v$ for complex numbers $c_v$. Then 
$$
|\lambda(w)| = |\sum_v c_v \lambda(v)| \leq \sum_v |c_v| \cdot |\lambda(v)| \leq \max_{v \in \mathcal B} \{ |\lambda(v)|\} \sum_v |c_v|. 
$$
By the Cauchy--Schwarz inequality we get 
$$
\sum_v |c_v| \leq \sqrt{\dim V} \cdot \sqrt{ \sum_v |c_v|^2} =\sqrt{\dim V} \cdot ||w||. 
$$
Consequently, we have showed
$$
|\lambda(w)| \leq \sqrt{\dim V} \cdot  \max_{v \in \mathcal B} \{ |\lambda(v)|\} \cdot ||w||. 
$$
This proves the assertion. 
\end{proof}

\subsection{The spectral decomposition for the height zeta function}
We now return to the proof of Theorem \ref{thm:semi-simple} by applying the above spectral analysis to the height zeta function.
For $g \in G(\Adele_k)$ we set 
\begin{equation} \label{def:HZF}
Z_\br(\sbf, g) = \sum_{ \gamma \in G(k)} \thorn(\gamma g) H(\sbf, \gamma g)^{-1}, \quad Z_\br(\sbf):=Z_\br(\sbf, e),
\end{equation}
where $e$ the identity element of $G$.
This is absolutely convergent for $\sbf \in \mathcal{T}_{C}$ for some sufficiently large $C$ (see \eqref{def:T_c}). For any fixed $\sbf$ in the domain of absolute convergence $Z_\br(\sbf, g)$ defines a continuous function which is bounded on $G(\Adele_k)$ and square integrable on $G(\Adele_k)$ (this can be shown using a minor variant of the proof of \cite[Prop.~2.3]{STBT_PGL2}, using the fact that $|\thorn| \leq 1$). As the height function is right invariant under $\K_0$, we therefore have
$$Z_\br(\sbf, g) \in L^2(G(k)\backslash G(\Adele_k))^{\K_0}.$$
\begin{lemma}
There exists $C > 0$ such that for $\sbf \in \mathcal{T}_{C}$ the function $Z_\br(\sbf, \cdot)$ satisfies the conditions of Theorem~\ref{lem:spectral-expansion-general}, and as such has a spectral decomposition.  
\end{lemma}
\begin{proof}
Since $|\thorn| \leq 1$, this is a consequence of \cite[Prop. 8.2]{JAMS}.  
\end{proof}
We therefore obtain the spectral expansion
\begin{equation} \label{eqn:spectral_expansion}
Z_\br(\sbf) = \sum_{\chi \in (G(\Adele_k)/G(k))^\wedge} \int_{G(\Adele_k)} H(\sbf,g)^{-1} \chi(g) \thorn(g) \mathrm{d}g + S^\flat(\sbf).
\end{equation}
Here, as in \cite[\S 8.1]{JAMS}, we denote by $S^\flat(\sbf)$ the contribution of the non-one dimensional representations to the spectral decomposition after putting $g = e$.

\subsection{Continuous and cuspidal spectrum}
We now handle the contribution $S^\flat$ from the continuous and cuspidal spectrum.
The analysis here is very similar to that in \cite{JAMS}, so we will be brief.

Let $\xi$ be an idempotent chosen as in \S \ref{bounds:matrix} in the global Hecke algebra such that $\xi * (H\cdot \thorn) = (H \cdot \thorn) * \xi = H$. With the notations of \S\ref{bounds:matrix} set $W = \otimes_{v \text{ arch.}} W_v$, $\xi_0 = \otimes'_{v \text{ non-arch.}} \xi_v$, and $\xi_W = \otimes_{v \text{ arch.}} \xi_v$. Let $P$ be a standard parabolic subgroup. For $\pi \in \Pi(M_P)$, $\chi \in \mathfrak X$, and 
$\phi \in \mathcal B(\pi)_\chi$, with $\phi * \xi = \phi$ we consider the integral 
$$\widehat{H}_\xi(\br,\sbf,E(\phi)) = \int_{G(\Adele_k)}\xi *(H(\sbf,\cdot)^{-1} \thorn(\cdot))(g) E(g, \phi)  \mathrm{d}g, \quad \sbf \in \mathcal{T}_{C},$$ 
for some large $C> 0$.
It is easy to see that for $\sbf \in \mathcal{T}_{C}$,  
$$
\widehat{H}_\xi(\br,\sbf,E(\phi)) = \int_{G(\Adele_k)}H(\sbf,g)^{-1} \thorn(g) (\xi * E)(g, \phi)  \mathrm{d}g
$$
Once we use Lemma \ref{lem:4.6}, the same argument as in the proof of \cite[Cor. 7.4]{JAMS}, which uses the bounds from \eqref{Xi}, gives the following theorem. 

\begin{theorem}\label{coro-7.4}
The function $\widehat{H}_\xi(\br,\sbf,E(\phi))$ has an analytic continuation to a function which is holomorphic on $\mathcal{T}_{-c}$, with $c>0$ as in \cite[Thm.~4.5]{JAMS}. For each integer $r> 0$, all $\varepsilon >0$, and every compact subset $U \subset \mathcal{T}_{-c + \varepsilon}$, there exists a constant $C = C(\varepsilon,r, U, \xi_0)$, independent of $\phi$, such that 
$$
|\widehat{H}_\xi(\br,\sbf,E(\phi))| \leq C \Lambda(\phi)^{-r} |E(e, \phi)|\sqrt{\dim \mathcal V_P(\pi)_{\chi, \K_0, W}} 
\max_{\varphi \in \mathcal B_P(\pi)_\chi \cap \mathcal V_P(\pi)_{\chi, \K_0, W}} \hspace{-25pt} \{ |E(e, \varphi)|\}. 
$$
for all $\sbf \in U$. 
\end{theorem} 
The following theorem is the analogue of  \cite[Thm.~8.3]{JAMS} in our context: 

\begin{theorem} \label{thm:continous_cuspidal_spectrum}
The function $S^\flat(\sbf)$ has an analytic continuation to a holomorphic function on $\mathcal{T}_{-c}$. 
\end{theorem}
\begin{proof}
Since $\thorn$ and $H$ are invariant on the left and right under the compact open subgroup $\K_v$ for each non-archimdean place $v$, there is an associated idempotent $\xi_0=\otimes_{v \text{ non-arch.}}' \xi_v$ in the Hecke algebra $\otimes_{v \text{ non-arch.}}' \mathcal H(G(k_v))$ such that $\xi_0 * (\thorn\cdot H) = (\thorn\cdot H) * \xi_0 = \thorn \cdot H $. By a theorem of Harish-Chandra \cite[Prop.~4.4.3.2]{Warner} we know that $\sum_{W \in \hat\K_\infty} \xi_W * (\thorn_\infty \cdot H_\infty)$ converges in the topology of $C^\infty(G_\infty)$ to $\thorn_\infty \cdot H_\infty$, c.f.~\cite[Appendix 2]{Warner} for a description of the topology.  Following the proof of \cite[Thm.~8.3]{JAMS} and after using Theorem \ref{coro-7.4} we need to show the convergence of 
$$
\sum_{\chi \in \mathfrak X}\sum_{P} n(\mathsf A_P)^{-1} \sum_{W \in \hat\K_\infty} \int_{\Pi(M_P)}\Big( \sum_{\phi
\in \mathcal B_P(\pi)_{\chi}\cap \mathcal V_P(\pi)_{\chi, \K_0, W}} 
\hspace{-25pt} \Lambda(\phi)^{-r} |E(e, \phi)|\sqrt{\dim \mathcal V_P(\pi)_{\chi, \K_0, W}} 
$$
\begin{equation}\label{equation}
\times \max_{\varphi \in \mathcal B_P(\pi)_\chi \cap \mathcal V_P(\pi)_{\chi, \K_0, W}} \hspace{-25pt}
 \{ |E(e, \varphi)|\}\Big)\, \mathrm{d}\pi
\end{equation}
for $r$ large.  The outermost summation is only over those classes $\chi$ for which there is a $\K_0$-fixed $\phi$ with $\Lambda(\phi) \ne 0$. Note for all $\phi, \phi' \in \mathcal B_P(\pi)_{\chi}\cap \mathcal V_P(\pi)_{\chi, \K_0, W}$, $\Lambda(\phi) = \Lambda(\phi')$. Let us denote this common value with $\Lambda(\pi, \chi, P, \K_0, W)$. We have 
$$
\sqrt{\dim \mathcal V_P(\pi)_{\chi, \K_0, W}}  \sum_{\phi
\in \mathcal B_P(\pi)_{\chi}\cap \mathcal V_P(\pi)_{\chi, \K_0, W}} \hspace{-25pt}  |E(e, \phi)|  \max_{\varphi \in \mathcal B_P(\pi)_\chi \cap \mathcal V_P(\pi)_{\chi, \K_0, W}}
 \hspace{-25pt} \{ |E(e, \varphi)|\}
$$
$$
\leq (\dim \mathcal V_P(\pi)_{\chi, \K_0, W})^{3/2} \max_{\varphi \in \mathcal B_P(\pi)_\chi \cap \mathcal V_P(\pi)_{\chi, \K_0, W}}
 \hspace{-25pt} \{ |E(e, \varphi)|^2\}
$$
$$
\leq (\dim \mathcal V_P(\pi)_{\chi, \K_0, W})^{3/2} \sum_{\phi
\in \mathcal B_P(\pi)_{\chi}\cap \mathcal V_P(\pi)_{\chi, \K_0, W}} \hspace{-25pt} |E(e, \phi)|^2.
$$
Consequently, the expression \eqref{equation} is majorized by 
$$
\sum_{\chi \in \mathfrak X}\sum_{P} n(\mathsf A_P)^{-1} \hspace{-5pt} \sum_{W \in \hat\K_\infty} \hspace{2pt} \int\limits_{\Pi(M_P)}
\hspace{-5pt} \left( \hspace{-5pt} \sum_{\,\,\,\phi
\in \mathcal B_P(\pi)_{\chi}\cap \mathcal V_P(\pi)_{\chi, \K_0, W}}  
\hspace{-30pt} \Lambda(\phi)^{-r}(\dim \mathcal V_P(\pi)_{\chi, \K_0, W})^{3/2} |E(e, \phi)|^2 \right) \, \mathrm{d}\pi
$$
which, by Proposition \ref{important-proposition} and \cite[Corollary 0.3]{Muller}, is convergent for large $r$. 
\end{proof}

\subsection{Automorphic characters}
We next handle the contribution to the spectral decomposition \eqref{eqn:spectral_expansion} coming from the automorphic
characters. We are interested in the following height integrals

\begin{equation} \label{def:character_integral}
	\widehat{H}(\br,\sbf,\chi) = \int_{G(\Adele_k)}H(\sbf,g)^{-1} \thorn(g) \chi(g)  \mathrm{d}g,
\end{equation}
where $\chi$ is an automorphic character of $G$. 
First note that $\thorn_v, \chi_v$, and $H_v$
take the constant value $1$ on $G(\OO_v)$ for all but finitely many $v$ (this follows from Lemma \ref{lem:K_v}).
We have also normalised
our measures so that $G(\OO_v)$ has volume $1$ for all but finitely many $v$, thus there is 
an Euler product decomposition
\begin{equation} \label{def:local_character_integral}
	\widehat{H}(\br,\sbf,\chi) = \prod_v \widehat{H}_v(\br,\sbf,\chi), \,\, 
	\widehat{H}_v(\br,\sbf,\chi_v) = \int_{G(k_v)} \hspace{-10pt}H_v(\sbf,g_v)^{-1} \thorn_v(g_v) \chi(g_v)  \mathrm{d}g_v.
\end{equation}
In the case where $\br$ is trivial, a regularisation
of such integrals was obtained in \cite[Thm.~7.1]{JAMS}, whereas in the toric case a regularisation 
was obtained in \cite[Lem.~5.10]{Lou13}. In this section, we obtain an analogue of these results in our setting.

\subsubsection{Local height integrals}

We let $\K_v$ be as in Lemma \ref{lem:K_v}.

\begin{lemma} \label{lem:local}
	Let $v$ be a place of $k$ and $\chi_v$ a character of $G(k_v)$.
	\begin{enumerate}
		\item The local height integral 
		$\widehat{H}_v(\br,\sbf,\chi_v)$ is absolutely convergent and holomorphic in the domain $\mathcal{T}_{-1}$.
		\label{item:converges}
		\item The local height integral $\widehat{H}_v(\br,\sbf,1)$ for the trivial character is non-zero
		for any real $\sbf$ in the domain $\mathcal{T}_{-1}$. \label{item:non-zero}
		\item If $v$ is non-archimedean and $\chi_v$ is non-trivial on $\K_v$, then $\widehat{H}_v(\br,\sbf,\chi_v) = 0$.
		\label{item:vanishes}

	\end{enumerate}
\end{lemma}
\begin{proof}
	As $|\thorn_v\chi_v| \leq 1$, part \eqref{item:converges} follows from the proof of \cite[Thm.~6.7]{JAMS}.
	Part \eqref{item:non-zero} follows simply from the fact that the integral of a continuous function that is positive on the support of a measure is non-zero.
	
	For \eqref{item:vanishes}, choose $\kappa_v \in \K_v$
	such that $\chi_v(\kappa_v) \neq 1$. 	Recall that $\thorn_v$ and $H_v$ are bi-$\K_v$-invariant, by the construction of $\K_v$ (cf.~Lemma \ref{lem:K_v}).
	Using this and the fact that $\mathrm{d}g_v$ is a Haar measure, we find that
	\begin{align*}
	\widehat{H}_v(\br,\sbf,\chi_v) 
	&= \int_{G(k_v)} H_v(\sbf,g_v\kappa_v)^{-1}
	\thorn_v(g_v\kappa_v) \chi_v(g_v\kappa_v)  \mathrm{d}g_v\\
	&=  \int_{G(k_v)} H_v(\sbf,g_v)^{-1}
	\thorn_v(g_v) \chi_v(g_v)\chi_v(\kappa_v)  \mathrm{d}g_v\\
	& = \chi_v(\kappa_v)\widehat{H}_v(\br,\sbf,\chi_v),
	\end{align*}
	whence $\widehat{H}_v(\br,\sbf,\chi_v) = 0$, as claimed.
\end{proof}

\begin{remark}
	Part \eqref{item:vanishes} of Lemma \ref{lem:local} is one of the crucial places in the argument where we use the fact that 
	the Brauer group elements under consideration are \emph{algebraic}. This will
	mean that only finitely many characters have non-zero height integral in the spectral decomposition \eqref{eqn:spectral_expansion}.
	For transcendental elements, many more height integrals appear (a similar phenomenon occurred in the case
	of toric varieties \cite[Rem.~5.5]{Lou13}).
\end{remark}

\subsubsection{Partial Euler products}
We shall regularise the height integrals using the partial Euler products from \cite[\S3.2]{Lou13}.
We briefly recall their definition and basic properties.

Let $\mathscr{X}$ be a finite group of Hecke characters of $k$ and let $\chi: \Adele^*/k^* \to S^1$ be a Hecke character of $k$. The partial Euler product
of interest to us is
\begin{equation}\label{def:partial_zeta_function}
	L_{\X}(\chi,s)=\prod_{\substack{v \in \Val(k) \\ \rho_v(\pi_v) = 1  \\ \forall \rho \in \X}} \left(1-\frac{\chi_{v}(\pi_{v})}{q_v^{s}}\right)^{-1},
	\quad \re s > 1,
\end{equation}
where the Euler product is only taken over those non-archimedean places $v$ for
which $\chi_v$ and $\rho_v$ are unramified for all $\rho \in \X$. Here $\pi_v$ denotes a uniformiser of $k_v$ and $q_v$
the size of the residue field.

\begin{lemma} \label{lem:partial}
	Assume that $\chi$ has finite order.
	If $\chi \not \in \X$, then $L_{\X}(\chi,s)$ admits a holomorphic
	continuation to $\re s = 1$. If $\chi \in \X$, then $(s-1)^{1/|\X|}L_{\X}(\chi,s)$ 
	admits a holomorphic continuation to the line $\re s =1$; in particular
	$L_{\X}(\chi,s)$  has a branch point singularity of order $-1/|\X|$ at $s=1$.
\end{lemma}
See \S\ref{sec:notation} for our conventions regarding holomorphicity and branch point singularities.
\begin{proof}
	In \cite[Lem.~3.2]{Lou13} it is shown that
	$$L_{\X}(\chi,s)^{|\X|} =G(\X,\chi,s) \prod_{\rho \in \X} L(\rho\chi,s), \quad \re s > 1,$$
	where $G(\X,\chi,s)$ is holomorphic and non-zero on $\re s > 1/2$.
	If $\chi\rho$ is non-trivial then the Hecke $L$-function
	$L(\rho\chi,s)$ is holomorphic and non-zero on some open
	neighbourhood of $\re s \geq 1$, so we can take its $|\X|$th-root in this domain.
	Otherwise we obtain the Dedekind zeta function $\zeta_k(s)$;
	in this case $((s-1)\zeta_k(s))^{1/|\X|}$ admits a holomorphic
	continuation to $\re s \geq 1$ and is non-zero at $s=1$
	(this follows from a minor adaptation of the arguments given in \cite[\S II.5.2]{Ten95}).
	The result follows.
\end{proof}

\subsubsection{Regularisation}
\label{subsec: regularisation}

We next recall how to associate a collection of Hecke characters to an automorphic character (see \cite[\S2.8]{JAMS}).
To construct these Hecke characters, one first performs a transfer to reduce to the quasi-split case $G'$ 
(see \S\ref{sec:transfer_Br_Ch}). By Lemma \ref{lem:WR}, the restriction of an automorphic
character $\chi$ of $G'$ to  $T'$ is a collection automorphic characters of the tori $\Res_{k_\alpha/k} \Gm$. But by the definition of the Weil restriction, an automorphic character of $\Res_{k_\alpha/k} \Gm$ is exactly an automorphic character of $\mathbb{G}_{\mathrm{m}, k_\alpha}$, i.e.~a Hecke character $\chi_\alpha$ over $k_\alpha$.
For $\alpha \in \A$, we denote by $\R_\alpha$ the collection
of Hecke characters induced by $\R$ from this construction. We also let $\kappa_\alpha$ be as in \eqref{def:kappa}.

\begin{theorem} \label{thm:characters}
	Let $\chi$ be an automorphic character of $G$ and let
	$\widehat{H}(\br,\sbf,\chi)$ be as in \eqref{def:character_integral}. Then we have 
	$$\widehat{H}(\br,\sbf,\chi) =  f(\sbf,\chi) \prod_{\alpha \in \A} L_{\R_\alpha}(\chi_\alpha,s_\alpha - \kappa_\alpha),
	\quad \text{for } \sbf \in \mathcal{T}_{0},$$
	where $f(\sbf,\chi)$ is holomorphic on $\mathcal{T}_{-\varepsilon}$ for some $\varepsilon > 0$
	and $f(\sbf, 1)$ is non-zero for any real $\sbf$ in this domain.
\end{theorem}
\begin{proof}
	It suffices to study the local height integrals $\widehat{H}_v(\br,\sbf,\chi_v)$.
	Character orthogonality yields
	$$\thorn_v(g_v) = \frac{1}{|\R|}\sum_{\rho \in \R} \rho_v(g_v).$$
	Applying this to our local height integrals, we obtain
	\begin{equation}	\label{eqn:Fourier_thorn_rho}
		\widehat{H}_v(\br,\sbf,\chi_v) = \frac{1}{|\R|}\sum_{\rho \in \R} \widehat{H}_v(0,\sbf,\rho_v \chi_v).
	\end{equation}	
	We can handle the $\widehat{H}_v(0,\sbf,\rho_v \chi_v)$ using \cite[Thm.~7.1]{JAMS}
	(note that whilst the proof of \emph{loc.~cit.} is correct there is a typo in the statement;
	the relevant $L$-functions need to be shifted by $-\kappa_\alpha$).
	
	We  now introduce some notation. Let $\A_v$ be the set of boundary divisors of $X_{k_v}$.
	For $\alpha_v \in \A_v$ and $\alpha \in \A$, we say that $\alpha_v \mid \alpha$ if 
	$D_{\alpha_v}$ is an irreducible component of $(D_\alpha)_{k_v}$; the set of such $\alpha_v$
	can be identified with the set of places of $k_{\alpha}$ which divide $v$.
	For $\alpha_v \in \A_v$, we denote by $k_{\alpha_v}$ the algebraic closure of $k_v$ in the residue field
	$\kappa(D_{\alpha_v})$.
	Let $f_{\alpha_v} = [k_{\alpha_v} : k_v]$ and let $\pi_{\alpha_v}$ be a uniformising parameter
	of $k_{\alpha_v}$. Denote by $s_{\alpha_v}$ the complex number $s_\alpha$,
	where $\alpha_v \mid {\alpha}$ (we define $\kappa_{\alpha_v}$ similarly).
	Then in the proof of \cite[Thm.~7.1]{JAMS} 
	(see in particular $(7.7)$ of \emph{ibid.}), it is shown that
	for all but finitely many $v$ we have
	$$\widehat{H}_v(0,\sbf,\chi) = \prod_{\alpha_v \in \A_v} 
	\left(1-\frac{\chi_{\alpha_v}(\pi_{\alpha_v})}{q_v^{f_{\alpha_v}(s_{\alpha_v}-\kappa_{\alpha_v})}}\right)^{-1}
	\left(1 + O_\varepsilon\left(\frac{1}{q_v^{1+\varepsilon}}\right)\right).$$

	Expanding this out, applying \eqref{eqn:Fourier_thorn_rho} and using character orthogonality, we obtain
	\begin{align}
		\widehat{H}_v(\thorn_v,\chi_v;-\sbf)
		&= 1+\frac{1}{|\R|}\sum_{\rho \in \R}\sum_{\alpha_v \in \A_v}
		\frac{\rho_{\alpha_v}(\pi_{\alpha_v})\chi_{\alpha_v}(\pi_{\alpha_v})}{q_v^{f_{\alpha_v}(s_{\alpha_v}-\kappa_{\alpha_v})}}+
		O_\varepsilon\left(\frac{1}{q_v^{1+\varepsilon}}\right) \nonumber \\
		&=\prod_{\substack{\alpha_v \in \A_v \\ \rho_{\alpha_v}(\pi_{\alpha_v}) = 1  \\ \forall \rho \in \R}}  
		\left(1-\frac{\chi_{\alpha_v}(\pi_{\alpha_v})}{q_v^{f_{\alpha_v}(s_{\alpha_v}-\kappa_{\alpha_v})}}\right)^{-1}
		\left(1 + O_\varepsilon\left(\frac{1}{q_v^{1+\varepsilon}}\right)\right). \label{eqn:Euler_factor}
	\end{align}
	Recalling that the set of $\alpha_v \mid \alpha$
	can be identified with the set of places of $k_{\alpha}$ dividing $v$, we can compare
	Euler products to get  the result.
	Finally, the non-vanishing of $f(\sbf, 1)$
	in the stated domain follows from the non-vanishing of the local
	height integrals from Lemma~\ref{lem:local}.
\end{proof}

\subsection{Continuation of the height zeta function}

\subsubsection{Residues and Hecke characters}
By Theorem \ref{thm:characters}, the order of the singularity of the height zeta function 
can be expressed in terms of the Hecke characters attached to $\br$. To deduce a result in terms of 
the original set of Brauer group elements $\br$, we need to relate these Hecke characters to 
the residues of $\br$ (namely obtain an analogue of \cite[Lem.~4.7]{Lou13} in our setting). We achieve
this by reducing to the case of toric varieties, as treated in \emph{loc.~cit.}, using the following.

\begin{proposition} \label{prop:residue_restriction}
	Let $f:V_1 \to V_2$ be a morphism of smooth varieties 
	over a field $F$ of characteristic $0$.
	Let $D_2 \subset V_2$ be a smooth irreducible divisor
	and assume that the scheme-theoretic preimage 
	$D_2 := f^{-1}(D_1)$ is also a smooth irreducible
	divisor. Let $V_i^{\circ}= V_i \setminus D_i$.
	Then the diagram
	$$
	\xymatrix{
		 \Br V_2^{\circ} \ar[d]^{\res_{D_2}} \ar[r] & \Br V_1^{\circ} \ar[d]^{\res_{D_1}} \\
		 \HH^1(D_2,\QQ/\ZZ) \ar[r] & \HH^1(D_1,\QQ/\ZZ),
		}
	$$
	obtained by pulling-back along $f$ and taking residues, is commutative.
\end{proposition}
\begin{proof}
	For the proof we use an alternative construction \cite[Rem.~3.3.2]{CT95}
	for the residue map coming from the Gysin exact sequence. 
	To do so we recall some facts about the Gysin exact sequence and 
	purity from \cite[\S 3]{CT95} and  \cite[\S1, \S2]{BO75}.
	
	We first work in slightly more generality than the statement of the proposition.
	Let $f:V_1 \to V_2$ be a morphism of $F$-varieties, let $Z_2 \subset V_2$
	be a closed subscheme and $Z_1 := f^{-1}(Z_2)$ its scheme-theoretic preimage.
	Let $n \in \NN$ and $V_i^{\circ}= V_i \setminus Z_i$.
	Then one has a commutative diagram with exact rows 
	\begin{equation} \label{seq:Gysin}
	\begin{split}
	\xymatrix{
\cdots \ar[r] & \HH^{i}(V_2,\mu_n) \ar[d] \ar[r] & \HH^{i}(V_2^\circ,\mu_n) \ar[d]\ar[r] & \HH^{i+1}_{Z_2}(V_2,\mu_n) \ar[d]\ar[r] &\cdots \\
\cdots \ar[r] & \HH^{i}(V_1,\mu_n) \ar[r] & \HH^{i}(V_1^\circ,\mu_n) \ar[r] & \HH^{i+1}_{Z_1}(V_1,\mu_n) \ar[r] & \cdots 
		}
	\end{split}		
	\end{equation}
	for \'etale cohomology with supports (see \cite[Def.~1.1]{BO75},
	\cite[\S2.1]{BO75}, and \cite[(3.3)]{CT95}).
	
	Let now $V$ be a smooth variety over $F$, let $D \subset V$ be a smooth irreducible divisor
	and $V^\circ = V \setminus D$.
	In this case cohomological purity \cite[(3.5) \& Thm.~3.4.1]{CT95} gives a functorial
	isomorphism $\HH^{i}_D(V,\mu_n) \cong \HH^{i-2}(D,\ZZ/n\ZZ)$. Applying
	this to the exact sequence \eqref{seq:Gysin},
	we obtain the  map
	\begin{equation} \label{eqn:Gysin}
		\HH^2( V^{\circ}, \mu_n) \to \HH^1(D,\ZZ/n\ZZ).
	\end{equation}
	To study this, consider the Kummer sequence
	$$ 1 \to \mu_n \to \Gm \to \Gm \to 1.$$
	Applying \'etale cohomology we obtain the exact sequence
	$$0 \to \Pic(V^{\circ})/n \to \HH^2( V^{\circ}, \mu_n) \to \Br( V^{\circ})[n] \to 0.$$
	The map in \eqref{eqn:Gysin} is trivial on $\Pic(V^{\circ})/n$, so we obtain
	 $\Br(V^{\circ})[n] \to \HH^1(D,\ZZ/n\ZZ)$. As $V$ is smooth 
	the Brauer group is torsion \cite[Prop.~6.6.7]{Poo17},
	thus we can combine these maps for all $n$ together to obtain a map
	\begin{equation} \label{eqn:CTresidue}
		\Br V^{\circ} \to \HH^1(D,\QQ/\ZZ).
	\end{equation}
	As explained in \cite[Rem.~3.3.2]{CT95}, this differs from the usual residue map
	by a sign.
	
	Let now $V_1,D_1,V_2,D_2$ be as in the statement of the proposition.
	As $D_1$ and $D_2$ are both smooth irreducible divisors,
	we may use \eqref{seq:Gysin} and \eqref{eqn:Gysin}
	to obtain a commutative diagram
	\begin{equation} \label{eqn:CTcommutes}
	\begin{split}
	\xymatrix{
		 \HH^2( V_2^{\circ}, \mu_n) \ar[d] \ar[r] & \HH^2( V_1^{\circ}, \mu_n) \ar[d] \\
		 \HH^1(D_2,\ZZ/n\ZZ) \ar[r] & \HH^1(D_1,\ZZ/n\ZZ).
		}
	\end{split}		
	\end{equation}
	As the Kummer sequence is functorial, combining 
	the above construction of the residue map \eqref{eqn:CTresidue} 
	with the commutativity of \eqref{eqn:CTcommutes} yields the result.
\end{proof}

Using this we obtain the following.

\begin{lemma} \label{lem:residues}
	The diagram
	$$
	\xymatrix{
		 \Br_e G \ar[d] \ar[r] & (G(\Adele_k)/G(k))^\sim \ar[d] \\
		 \displaystyle{\bigoplus_{\alpha \in \A}} \HH^1(k_\alpha,\QQ/\ZZ) \ar[r] & \displaystyle{\bigoplus_{\alpha \in \A}} (\Gm(\Adele_{k_\alpha})/\Gm(k_\alpha))^\sim}
	$$
	commutes up to sign. Here the top arrow comes from Lemma \ref{lem:Br_char}.
	The left arrow is given by the residues maps attached to the boundary divisors $D_\alpha \subset X$.
	The bottom arrow is an isomorphism, and comes from class field theory and the identification 
	$\HH^1(k,\QQ/\ZZ) = \Hom(\Gal(\overline{k}/k), \QQ/\ZZ)$. 
	The right arrow is the Hecke character construction explained in \S\ref{subsec: regularisation}.
\end{lemma}
\begin{proof}
By applying a transfer on Brauer groups and automorphic characters, we may assume that $G$ is quasi-split
(see Lemma \ref{lem:transfer_residues} and Lemma \ref{lem:transfer_characters}).

Let $B$ be a Borel subgroup of $G$ and let $x_0$ be a point on $G$ such that $Bx_0B$ is an open orbit on $G$.
Let $T \subset B \subset G$ be a maximal torus; this has the form \eqref{eqn:Weil_restriction}.
We denote the closure of $T$ inside $X^\circ$ by $Z$; this is a smooth toric variety by Lemma \ref{lem:spherical}. Then Lemma~\ref{lem:spherical} tells us that for any $\alpha \in \A$ the scheme theoretic intersection $E_\alpha = D_\alpha \cap Z$ is integral and has the same field of constants as $D_\alpha$ (namely $k_\alpha$).
Consider the following diagram:
$$
	\xymatrix{
      \Br_e G  \ar[r] \ar[d]^{\res_\alpha} &  \Br_e T \ar[r] \ar[d]^{\res_\alpha} & (T(\mathbb A_k)/T(k))^\sim \ar[d] \\
		  \HH^1(k_\alpha,\QQ/\ZZ)    \ar[r]&   \HH^1(k_\alpha,\QQ/\ZZ)  \ar[r]  & (\Gm(\Adele_{k_\alpha})/\Gm(k_\alpha))^\sim
		 }
$$
Applying Proposition \ref{prop:residue_restriction} to the inclusion $Z \to X$
and the open subset of $X$ given by removing the singular locus of $D_\alpha$ and $E_\alpha$, we see that the left hand square commutes.
That the right hand square commutes up to sign is shown in \cite[Lem.~4.7]{Lou13} (this is proved using \cite[Lem.~9.1]{San81}). 
Note that \emph{loc.~cit.} is stated for complete toric varieties, but we can apply this to a compactification of $Z$.
The result now easily follows.
\end{proof}

\subsubsection{The height zeta function}
We now address the main term in the spectral expansion \eqref{eqn:spectral_expansion}. By Theorem \ref{thm:continous_cuspidal_spectrum}, this has the shape
$$\sum_{\chi \in (G(\Adele_k)/G(k))^\wedge} \widehat{H}(\br,\sbf,\chi),$$
where $\widehat{H}(\br,\sbf,\chi)$ is as in \eqref{def:character_integral}.
We first show that this sum is in fact finite.

\begin{lemma} \label{lem:vanish}
	Let $\K_0$ be as at the end of \S \ref{sec:heights}.
	Let
	$$ \mathfrak{X}_{\K_0} = \{ \chi \in (G(\Adele_k)/G(k))^\wedge: \chi(\K_0) = 1\}.$$
	Then $ \mathfrak{X}_{\K_0}$ is finite and for all automorphic characters 
	$\chi \not \in \mathfrak{X}_{\K_0}$ we have
	$\widehat{H}(\br,\sbf,\chi) = 0.$
\end{lemma}
\begin{proof}
	By Lemma \ref{lem:K_v}, the group $\K_0$ has finite index inside some maximal compact subgroup of 
	the group $G(\Adele_{k,f})$ of finite adeles.
	The finiteness of $\mathfrak{X}_{\K_0}$ thus follows from  \cite[Lem.~4.7]{GMO08}.
	The second part  was shown in Lemma \ref{lem:local}.
\end{proof}

We now restrict our attention to a complex line inside $\CC^\A$ corresponding to a given choice of line bundle. 
Let $L \in (\Pic X) \otimes \RR$ be the class of a big divisor. By Proposition~\ref{prop:Picard} we may write uniquely
$L = \sum_{\alpha \in \A} a_\alpha D_\alpha$ for some $a_\alpha \in \RR_{>0}$. Let
$$a(L) = \max_{\alpha \in \A} \frac{1 + \kappa_\alpha}{a_\alpha}, \quad
\A(L) = \left\{ \alpha \in \A : \frac{1 + \kappa_\alpha}{a_\alpha} = a(L)\right\},
\quad m_\br(L) = \sum_{\alpha \in \A(L)} \frac{1}{|\R_\alpha|}.$$
One easily checks using Proposition \ref{prop:Picard}
that $a(L)$ agrees with the $a$-constant from the Batyrev-Manin conjecture \cite[Def.~2.1]{BM90}.
We also let $\underline{a} = (a_{\alpha})_{\alpha \in \A}$ and consider the finite group
\begin{equation} \label{def:characters_L}
	\mathfrak{X}_{\K_0}(L) = \{ \chi \in \mathfrak{X}_{\K_0} : 
	 \chi_\alpha \in \R_\alpha\,  \forall \alpha \in \A(L)\}
\end{equation}
of automorphic characters (note that $\R \subset \mathfrak{X}_{\K_0}(L)$). The relevant height zeta function is given by 
\begin{equation} \label{eqn:HZF_L}
	Z_{\br,L}(s) := Z_\br(s\underline{a}),
\end{equation}
where $Z_\br$ is as in \eqref{def:HZF}.
Our result on $Z_{\br,L}$ is the following.

\begin{theorem} 	\label{thm:HZF}
	We have 
	$$Z_{\br,L}(s) = (s-a(L))^{-m_\br(L)}g(s) + \sum_{i \in I} (s-a(L))^{-\lambda_i}g_i(s) + h(s)
	, \quad	\re s > a(L),$$
	for some finite index set $I$, some rational numbers $0<\lambda_i < m_\br(L)$,
	and some functions $g,g_i,h$ which are holomorphic 
	on $\re s \geq a(L)$. 	
	Furthermore, we have 
	\begin{equation*}
		\lim_{s \to a(L)}
		(s-a(L))^{m_\br(L)} Z_{\br,L}(s)
		=  \lim_{s \to a(L)}
		(s-a(L))^{m_\br(L)}\sum_{\mathclap{\chi \in \mathfrak X_{\K_0}(L)}}\,\,\,\,\,
		\int_{G(\Adele_k)_\br} \hspace{-15pt} H(s\underline{a}, g)^{-1}\chi(g) \, \mathrm{d} g,
	\end{equation*}
	and this limit is a positive real number. In particular $Z_{\br,L}(s)$ admits a branch point singularity of order
	$-m_\br(L)$ at $s=a(L)$.
\end{theorem}
\begin{proof}
	It follows from Theorem \ref{thm:continous_cuspidal_spectrum} that the contribution to the spectral decomposition
	\eqref{eqn:spectral_expansion} from the continuous and cuspidal spectrum is holomorphic on
	$\re s \geq a(L)$. Hence Lemma \ref{lem:partial}, Theorem \ref{thm:characters}, Lemma \ref{lem:residues}
	and Lemma \ref{lem:vanish} 	imply that it is the characters in $\mathfrak{X}_{\K_0}(L)$
	which give rise to the singularity of highest order,
	and that the height integrals have a branch point singularity of order
	at least $-m_\br(L)$ at $s=a(L)$. (Here the lower order terms $(s-a(L))^{-\lambda_i}$
	come from those characters $\chi \in \mathfrak{X}_{\K_0}\setminus \mathfrak{X}_{\K_0}(L)$).
	
	To finish it suffices to show that the leading constant does not vanish, which we do using a 
	variant of the proof of \cite[Lem.~5.14]{Lou13}.
	By character orthogonality, we see that the limit in the statement equals
	\begin{equation} \label{eqn:limit}
		|\mathfrak{X}_{\K_0}(L)|\lim_{s \to a(L)} (s-a(L))^{m_\br(L) }
		\int_{G(\Adele_k)_{\br}^{\mathfrak{X}_{\K_0}(L)}} H(s\underline{a},g)^{-1} \mathrm{d}g.
	\end{equation}
	Let $\br(L) \subset \Br_e G$ be the finite group  attached to 
	the automorphic characters $\mathfrak{X}_{\K_0}(L)$ by Corollary \ref{cor:Br_adjoint}.
	The expression \eqref{eqn:limit} is bounded below by
	$$\lim_{s \to a(L)} (s-a(L))^{m_\br(L) }
	\int_{G(\Adele_k)_{\br(L)}} H(s\underline{a},g)^{-1} \mathrm{d}g.$$
	However this integral is exactly equal to $\widehat{H}(\br(L),s\underline{a},1)$. It thus suffices to show that 
	$$\lim_{s \to a(L)} (s-a(L))^{m_\br(L)} \widehat{H}(\br(L),s\underline{a},1) \ > 0.$$
	As $a(L)$ is independent of $\br$, this will follow from 
	Theorem \ref{thm:characters} provided we show that
	$$
		m_\br(L) = m_{\br(L)}(L).
	$$
	However this is clear from the definition \eqref{def:characters_L} of $\mathfrak{X}_{\K_0}(L)$.
	This completes the proof.
\end{proof}

\subsection{The asymptotic formula}
We now come to our main theorem, of which Theorem \ref{thm:semi-simple} is a special case. 
For completeness we give the full statement.

\begin{theorem}
\label{thm:arbitrarybiglinebundles}
	Let $G$ be an adjoint semi-simple algebraic group over a number field $k$ with wonderful compactification $X$.
	Let $\br \subset \Br_1 G$ be a finite subgroup of algebraic Brauer group elements. 
	Assume that $G(k)_\br \neq \emptyset$. Let $L$ be a big line bundle on $X$.
	Then for all height functions $H_L$ associated with some choice of smooth adelic metric on $L$,
	we have
	$$N(G,H_L,\br,B) \sim c_{X,\br,H_L} B^{a(L)} \frac{(\log B)^{b(L)-1}}{(\log B)^{\Delta_X(L,\br)}},
	 \quad \text{as } B \to \infty,$$
	for some $c_{X,\br,H_L} > 0$. Here 
	\begin{itemize}
	\item 
	$a(L) = \inf\{ a \in \RR :  aL + K_X \in \Lambda_{\mathrm{eff}}(X)\}.$
	\item  $b(L)$ is the codimension of the minimal face of $\Lambda_{\mathrm{eff}}(X)$ containing 
	$a(L)L + K_X.$
	\item  	$$\Delta_X(L,\br)=\sum_{D \subset X}\left(1 - \frac{1}{|\res_D(\br)|}\right),$$
   	where the sum is over those divisors $D$ which do not appear in the support of any effective $\mathbb Q$-divisor $\mathbb Q$-linearly equivalent to $a(L)L + K_X$.
	\end{itemize}

\end{theorem}
The factors $a(L), b(L)$ which appear are the usual factors from the 
Batyrev-Manin conjecture \cite[Conj.~C']{BM90}. The factor $\Delta_X(L,\br)$ is new
and takes into account the collection of Brauer group elements $\br$.
\begin{proof}[Proof of Theorem \ref{thm:arbitrarybiglinebundles}]
	By Remark \ref{rem:heights}, the height zeta function attached to $H_L$ may be written in the form
	\eqref{eqn:HZF_L}. Given Theorem \ref{thm:HZF}, we may thus 
	apply Delange's Tauberian theorem \cite[Thm.~III]{Del54}
	to deduce an asymptotic formula of the shape 
	$$N(G,H_L,\br,B) \sim c_{X,\br,H_L} B^{a(L)} (\log B)^{m_\br(L)-1},
	 \quad \text{as } B \to \infty,$$
 	for some $c_{X,\br,H_L} > 0$.
 	To address $m_\br(L)$, by Proposition \ref{prop:Picard} the effective cone is simplicial and generated by the $D_\alpha$,
 	hence $b(L) = \#\A(L)$.
 	Moreover, it follows from Lemma~\ref{lem:residues} that $|\R_\alpha| = |\res_{D_\alpha}(\br)|$.
 	We find that
	$$
		m_\br(L) = \sum_{\alpha \in \A(L)} \frac{1}{|\R_\alpha|} 
		= b(L) - \sum_{\alpha \in \A(L)}\left( 1- \frac{1}{|\res_{D_\alpha}(\br)|}\right).
	$$
	The adjoint divisor here is given by 
	\begin{align*}
		a(L)L + K_X &=  \sum_{\alpha \in \A} (a(L)a_\alpha  - (1 + \kappa_\alpha))D_\alpha 
		=  \sum_{\alpha \notin \A(L)} (a(L)a_\alpha  - (1 + \kappa_\alpha))D_\alpha
	\end{align*}
	where the coefficients are positive. From this, it easily follows that 
	$$\A \setminus \A(L) = \{ \alpha \in \A: \exists D \sim a(L)L+K_X, \, D_\alpha \subset \mathrm{Supp}(D)\},$$
	as required. 
\end{proof}

\subsubsection{The case of the anticanonical divisor}
We now make  the above more explicit in the case of the anticanonical divisor. The answer we obtain is similar
to the case of anisotropic tori \cite[Thm.~5.15]{Lou13}. The calculation of the leading constant requires some of the theory
of subordinate Brauer group elements, as defined in \cite[\S2.6]{Lou13}, together with the 
Tamagawa measure $\tau_{\br}$ introduced in \cite[\S5.7]{Lou13}.

We first describe the Tamagawa measure $\tau_{\br}$. This is defined in a similar manner to Peyre's Tamagawa
measure \cite[\S2]{Pey95}, except that different convergence factors are used. 
Choose Haar measures $\mathrm{d}x_v$ on each $k_v$ such that $\vol(\OO_v)=1$ for all but finitely
many $v$. These thus give rise to a Haar measure $\mathrm{d}x$ on $\Adele_k$; we choose
our Haar measures so that $\vol(\Adele_k/k)=1$ with respect to the induced quotient measure.
Let $\omega$ be a left invariant top degree
differential form on $G$. By a classical construction \cite[\S2.1.7]{CLT10}, for any place $v$ of $k$
we obtain a measure $|\omega|_v$ which depends on the choice of $\mathrm{d}x_v$. Peyre's local Tamagawa measure 
is then given by
$$\tau_v = \frac{|\omega|_v}{\| \omega \|_v}.$$
For the convergence factors, we use the following virtual integral
Artin representation
$$\Pic_\br(\overline{X}) = \Pic(\overline{X}) - \sum_{\alpha \in \A} 
\left(1 - \frac{1}{|\res_{D_\alpha}(\br)|}\right) \Ind_{k_\alpha/k} \ZZ,$$
with associated virtual Artin $L$-function $L(\Pic_\br(\overline{X})_\CC,s)$ (cf.~\cite[\S5.7.1]{Lou13}). 
The relevant measure is then defined to be
$$\tau_\br = L^*(\Pic_\br(\overline{X})_\CC,1)
\prod_v \left( \tau_v \cdot L_v(\Pic_\br(\overline{X})_\CC,1)^{-1}\right).$$
Here
$$L^*(\Pic_\br(\overline{X})_\CC,1) = \lim_{s \to 1}(s-1)^{\rho(X) - \Delta_X(\br)}L(\Pic_\br(\overline{X})_\CC,s)$$ and  $L_v(\Pic_\br(\overline{X})_\CC,s)$ denotes the corresponding local Euler factor when $v$
is non-archimedean, and $L_v(\Pic_\br(\overline{X})_\CC,s) = 1$ otherwise. Note that we do not include
a discriminant factor like Peyre, as we have normalised our measures so that $\vol(\Adele_k/k)=1$.
It follows from Theorem \ref{thm:characters} that these are a family of convergence factors
(see the proof of Theorem \ref{thm:leading_constant} for details). 

\begin{theorem} \label{thm:leading_constant}
	Under the same assumptions as Theorem \ref{thm:semi-simple}, we have
	$$N(G,H,\br,B) \sim c_{X,\br,H} B \frac{(\log B)^{\rho(X)-1}}{(\log B)^{\Delta_X(\br)}},
	 \quad \text{as } B \to \infty.$$
	Here
	$$c_{X,\br,H}=\frac{|\Sub(X,\br)/\Br k| \cdot \tau_{\br}(G(\Adele_k)_{\br}^{\Sub(X,\br)})}
	{|\Pic G| \cdot \Gamma(\rho(X) - \Delta_X(\br)) \cdot
	\prod_{\alpha \in \A} (1+\kappa_\alpha)^{1/|\res_{D_\alpha}(\br)|}},$$
	where $\Gamma$ is the usual Gamma-function, the $\kappa_\alpha$ are as in 
	\eqref{def:kappa}, and 
	$$\Sub(X,\br) = \{b \in \Br G: \res_D(b) \in \langle \res_D(\br) \rangle\, \forall D \in X^{(1)} \},$$
	denotes the associated group of subordinate Brauer group elements.
	We also denote by $G(\Adele_k)_{\br}^{\Sub(X,\br)}$ the subset of $G(\Adele_k)_{\br}$ othogonal to 
	$\Sub(X,\br)$ with respect to the global Brauer pairing \eqref{def:global_Brauer_pairing}.
\end{theorem}
\begin{proof}
	The asymptotic formula follows from Theorem \ref{thm:arbitrarybiglinebundles}, as we have
	$a(-K_X)=1$ and $b(-K_X) = \rho(X)$. One also easily sees that $\Delta_X(-K_X, \br) = \Delta_X(\br)$,
	as the adjoint divisor is trivial in this case.
	
	It thus suffices to calculate the leading constant $c_{X,\br,H}$,
	which we do using the expression from Theorem \ref{thm:HZF}. An application of Delange's Tauberian theorem \cite[Thm.~III]{Del54}	shows that
	$$c_{X,\br,H} = 		\frac{|\mathfrak{X}_{\K_0}(-K_X)|}{
	\Gamma(\rho(X) - \Delta_X(\br))}
	\lim_{s \to 1} (s-1)^{\rho(X) -\Delta_X(\br)}
		\int_{G(\Adele_k)_{\br}^{\mathfrak{X}_{\K_0}(-K_X)}} H(g)^{-s} \mathrm{d}g.$$
	
	Note that we have $\A(-K_X) = \A$. It therefore follows easily from 
	Corollary \ref{cor:Br_adjoint} and Lemma \ref{lem:residues} that
	$$\mathfrak{X}_{\K_0}(-K_X) \cong \Sub(X,\br) \cap \Br_e G,$$
	and thus $|\mathfrak{X}_{\K_0}(-K_X)| = |\Sub (X,\br)/\Br k|$ (note that $\Sub(X,\br) \subset \Br_1 X$ by \cite[Lem.~2.14]{Lou13}).

	Next, by Proposition \ref{prop:Picard} 
	we have $\Pic(\overline{X})_\CC \cong \oplus_{\alpha \in \A} \Ind_{k_\alpha/k} \CC$,
	hence the associated virtual Artin $L$-function has the form
	\begin{equation} \label{eqn:Artin_L}
		L(s, \Pic_\br(\overline{X})) = \prod_{\alpha \in \A} \zeta_{k_\alpha}(s)^{1/|\res_{D_\alpha}(\br)|}.
	\end{equation}
	Using this, we find that
	\begin{align*}
		& \lim_{s \to 1} (s-1)^{\rho(X) - \Delta_X(\br) }
		\int_{G(\Adele_k)_{\br}^{\mathfrak{X}_{\K_0}(-K_X)}} H(g)^{-s} \mathrm{d}g \\
		= & \lim_{s \to 1} (s-1)^{\rho(X) - \Delta_X(\br) } 
		\frac{\prod_{\alpha \in \A} \zeta_{k_\alpha}((\kappa_\alpha + 1)s - \kappa_\alpha)^{1/|\res_{D_\alpha}(\br)|}}{\prod_{\alpha \in \A} \zeta_{k_\alpha}((\kappa_\alpha + 1)s - \kappa_\alpha)^{1/|\res_{D_\alpha}(\br)|}}
		\int_{G(\Adele_k)_{\br}^{\Sub(X,\br)}} H(g)^{-s} \mathrm{d}g \\
		= & \frac{ L^*(\Pic_\br(\overline{X})_\CC,1)}
		{\prod_{\alpha \in \A} (1+\kappa_\alpha)^{1/|\res_{D_\alpha}(\br)|}}
		\int_{G(\Adele_k)_{\br}^{\Sub(X,\br)}} \prod_v L_v(\Pic_\br(\overline{X})_\CC,1)^{-1}
		H(g)^{-1} \mathrm{d}g.
	\end{align*}
	Here one justifies taking the limit inside the integral
	using the dominated convergence theorem, which applies due to the explicit expressions for the local
	height integrals given in the proof of Theorem \ref{thm:characters} (cf.~\eqref{eqn:Euler_factor}).
	
	It remains to express the above integral in terms of the Tamagawa measure $\tau_\br$.
	Recall that $\mathrm{d}g$ is a choice of Haar measure on $G(\Adele_k)$ normalised so that $\vol(G(\Adele_k)/G(k))=1$.
	However, as $G$ is semi-simple, the measure $|\omega| = \prod_v|\omega|_v$ is the classical (i.e.~Weil's) Tamagawa
	measure on $G(\Adele_k)$. Denote by $\omega(G)$ the 
	Tamagawa number of $G$, i.e.~the measure of $G(\Adele_k)/G(k)$ with respect to $|\omega|$.
	A theorem of Ono \cite{Ono65} (see also
	\cite[Thm.~10.1]{San81}) states that
	$$\omega(G) = \frac{|\Pic G|}{|\Sha(G)|}.$$

	However, as $G$ is adjoint we have $\Sha(G) = 0$ \cite[Cor.~5.4]{San81}. Thus
	$\omega(G) = |\Pic G|$ and so by properties of Haar measures we obtain 
	\begin{equation} \label{eqn:measures}
		\mathrm{d} g = \frac{1}{|\Pic G|} |\omega|.
	\end{equation}	
	Using this we obtain
	\begin{align*}
		&L^*(\Pic_\br(\overline{X})_\CC,1)
		\int_{G(\Adele_k)_{\br}^{\Sub(X,\br)}} \prod_v L_v(\Pic_\br(\overline{X})_\CC,1)^{-1}
		H(g)^{-1} \mathrm{d}g  \\
		& = \frac{1}{|\Pic G|} \tau_\br\left(G(\Adele_k)_{\br}^{\Sub(X,\br)}\right).
	\end{align*}
	Combining everything together gives the result.
\end{proof}

\begin{remark}
	Let us consider the leading constant in Theorem \ref{thm:leading_constant}
	in the classical case $\br = \{0\}$, and explain how one recovers Peyre's constant
	$c_{\textrm{Peyre}} = \alpha(X)\beta(X)\tau(X)$ \cite[Def.~2.5]{Pey95}.
	Clearly $\tau_\br = \tau$ is Peyre's Tamagawa measure in this case. We also
	have
	$$\Sub(X,0)= \Br X.$$
	However $\beta(X)=|\Sub(X,0)/\Br k| = |\HH^1(k, \Pic \overline{X})| =1$ for wonderful compactifications 
	(see the proof of \cite[Thm.~9.2]{JAMS}).
	Using the explicit description of the effective cone given in Proposition \ref{prop:Picard},
	a simple calculation also shows that Peyre's effective cone constant
	is equal to
	$$\alpha(X) = \frac{1}{|\Pic G| \cdot (\rho(X) - 1)! \cdot \prod_{\alpha \in \A}(1 + \kappa_\alpha)}.$$
	This agrees with the denominator in Theorem \ref{thm:leading_constant} when $\br = \{0\}$.
\end{remark}

\section{The wonderful compactification of \texorpdfstring{$\PGL_n$}{Lg}} \label{sec:PGLn}
We now prove Theorem \ref{thm:PGL} using Theorem \ref{thm:arbitrarybiglinebundles}.

\subsection{The construction of the wonderful compactification}

We have the standard compactification of $\mathrm{PGL}_n$ in 
\[
X^{(0)}= \mathbb P^{n^2-1} = \mathrm{Proj}(k[x_{i, j}:{1 \leq i, j \leq n}]).
\]
This compactification is bi-equivariant, but not a wonderful  if $n\geq 3$. We construct the wonderful compactification of $\mathrm{PGL}_n$ by blowing up $X^{(0)}$. To this end, we consider the sequence of loci:
\[
Y^{(0)}_2 \hookrightarrow Y^{(0)}_3 \hookrightarrow \cdots \hookrightarrow Y^{(0)}_{n} \hookrightarrow X^{(0)},
\]
where $Y^{(0)}_r$ is the vanishing locus of the determinants of $r\times r$ minors of $(x_{i, j})_{1\leq i, j \leq n}$.
The wonderful compactification is given by blowing-up each of the loci $Y^{(0)}_r$ in turn, for $2 \leq r \leq n-1$.

Namely, we inductively construct $X^{(m)}$ and $Y^{(m)}_r$ for $1 \leq m \leq n-2$ as follows:
the variety $X^{(m)}$ is the blow up of $X^{(m-1)}$ along $Y^{(m-1)}_{m+1}$. We let $Y^{(m)}_r$ be the strict transform of $Y^{(m-1)}_r$, except when $r = m+1$ where we let $Y^{(m)}_{m+1}$ be the pullback of $Y^{(m-1)}_{m+1}$.
Let $X = X^{(n-2)}$ and $Y_r = Y_r^{(n-2)}$.
Note that $Y_{n}$ is the strict transform of the divisor $D: \det_n = 0 \subset X^{(0)}$.

\begin{proposition}{\cite[\S8]{Kau00}} \label{prop:wonderful}
The variety $X$ is the wonderful compactification of $\mathrm{PGL}_n$ with boundary components $Y_r$ $(2 \leq r \leq n)$. 
\end{proposition}

\subsection{Proof of Theorem~\ref{thm:PGL}}
Let $b \in \Br U$ be such that $U(\QQ)_b \neq \emptyset$. By Proposition~\ref{prop:PGL_n} we have $b \in \Br_1 U$.
Let $L$ be the pullback of the hyperplane class via $\pi : X \rightarrow X^{(0)} = \PP^{n^2-1}$.
As shown in \cite[\S2]{HTT15} the invariants $a(L), b(L)$ are birational invariants, so computing these  on $X^{(0)}$ gives  $a(L)=n^2$ and $b(L) = 1$.
Next smooth varieties have terminal singularities, hence
$$\pi^*(-K_{X^{(0)}}) = -K_X + \sum_{r=2}^{n-1}a_rY_r$$
for some $a_r > 0$. It follows that $Y_n$ is the only boundary component which does not appear the effective divisor linearly equivalent to $a(L)L + K_X$. So by Theorem~\ref{thm:arbitrarybiglinebundles} it suffices to show that the residue of $b$ along $Y_{n}$ has order $|b|$.   However, as $Y_{n}$ is the strict transform of $D$, we have $|\res_{Y_{n}}(b)| = |\res_D(b)|$. Since the residue map is injective in this case (see \eqref{eqn:res_PGL_n}), we find that $|\res_D(b)|= |b|$, as required. \qed


\begin{thebibliography}{xx}

\bibitem{Arthur78}
J. Arthur, 
{A trace formula for reductive groups I. Terms associated to classes in $G(\QQ)$}.
\emph{Duke Math. Journal}, {\bf 45} (1978), 911--952. 

\bibitem{Arthur79}
J. Arthur, 
{Eisenstein series and the trace formula.}
\emph{Automorphic forms, representations and L-functions} (Proc. Sympos. Pure Math., Oregon State Univ., Corvallis, Ore., 1977), Part 1, pp. 253--274, Proc. Sympos. Pure Math., XXXIII, Amer. Math. Soc., Providence, R.I., 1979.

\bibitem{BM90}
V. V. Batyrev, Y. I. Manin, 
{Sur le nombre des points rationnels
de hauteur born\'{e}e des vari\'{e}t\'{e}s alg\'{e}briques}.
\emph{Math. Ann.}, {\bf 286} (1990), 27--43.

\bibitem{BT95}
V. V. Batyrev, Y. Tschinkel, 
{Rational points of bounded height on compactifications of anisotropic tori}.
\emph{Int. Math. Res. Not.}, {\bf12} (1995), 591--635.

\bibitem{BO75}
Bloch, S., Ogus, A. 
{Gersten's conjecture and the homology of schemes.}
\emph{Ann. Sci. \'{E}cole Norm. Sup.} (4) {\bf 7} (1975), 181--201.

\bibitem{Bri07}
M. Brion,
{The total coordinate ring of a wonderful variety}.
\emph{J. Algebra} {\bf313} (2007), no. 1, 61--99. 

\bibitem{CLT10}
A. Chambert-Loir, Y. Tschinkel,
{Igusa integrals and volume asymptotics in analytic and adelic geometry}.
\emph{Confluentes Mathematici}, {\bf2}(3) (2010), 351--429.

\bibitem{CT95}
J.-L. Colliot-Th\'el\`ene,
Birational invariants, purity and the Gersten conjecture, in \emph{K-Theory and Algebraic Geometry: Connections with Quadratic Forms and Division Algebras}, AMS Summer Research Institute, Santa Barbara 1992, Proceedings of Symposia in Pure Mathematics 58, Part I (1995) 1-64.

\bibitem{DCP83}
{C. De Concini, C. Procesi},
{Complete symmetric varieties}.
\emph{Lecture Notes in Math.} 996 {\it Invariant theory ({M}ontecatini, 1982)} (1983) 1--44.

\bibitem{Del54}
{H. Delange}, 
{G\'{e}n\'{e}ralisation du th\'{e}or\`{e}me de Ikehara}.
\emph{Ann. Sci. \'{E}c. Norm. Sup\'{e}r.} (3) {\bf71} (1954), 213--242.

\bibitem{Den88}
{C. Deninger}, 
{A proper base change theorem for nontorsion sheaves in \'{e}tale cohomology}.
\emph{J. Pure Appl. Algebra} {\bf50} (1988), no. 3, 231--235.

\bibitem{EJ08}
{S. Evens, B. F. Jones},
{On the wonderful compactification}. (2008),
\texttt{arXiv:0801.0456}. Lecture notes based on lectures given at HKUST and Notre Dame.

\bibitem{Flath} 
{D. Flath}, 
{Decomposition of representations into tensor products}. 
\emph{Automorphic forms, representations and L-functions},
Proc. Sympos. Pure Math., {\bf33} (1979), Part 1, 179--183. 

\bibitem{HTT15}
{B. Hassett, S. Tanimoto, Y. Tschinkel},
{Balanced line bundles and equivariant compactifications of homogeneous spaces.}
\emph{Int. Math. Res. Not.}, {\bf 15}, (2015), 6375--6410.

\bibitem{Hoo93}
{C. Hooley}, 
{On ternary quadratic forms that represent zero}.
\emph{Glasgow Math. J.} {\bf35} (1993), no. 1, 13--23.

\bibitem{Hoo07}
{C. Hooley}, 
{On ternary quadratic forms that represent zero. II}.
\emph{J. reine angew. Math.} {\bf602} (2007), 179--225.

\bibitem{Ive76}
{B. Iversen}, 
{Brauer group of a linear algebraic group}.
\emph{J. Algebra} {\bf42} (1976), 295--301.

\bibitem{Jacquet-Langlands}
{H. Jacquet, R. P. Langlands}
\textit{ Automorphic Forms on $\GL(2)$}. 
Lecture Notes in Mathematics, Vol. 114. Springer-Verlag, Berlin-New York, 1970.

\bibitem{Lan08}
{E. Landau}, 
{\"{U}ber die Einteilung der positiven ganzen Zahlen in vier Klassen nach der Mindestzahl der
zu ihrer additiven Zusammensetzung erforderlichen Quadrate}.
\emph{Arch. der Math. und Physik} (3), {\bf13} (1908), 305--312.

\bibitem{Langlands65}
{R. P. Langlands}, 
{Eisenstein series}. 
\emph{Algebraic groups and discontinuous subgroups},
Proc. Sympos. Pure Math., Boulder, Colo., 1965, 235--252.

\bibitem{Langlands76}
R. P. Langlands, 
\emph{On the functional equations satisfied by Eisenstein series}. 
Lecture Notes in Mathematics, Vol. 544. Springer-Verlag, 1976.	

\bibitem{Lou13}
{D. Loughran}, 
{The number of varieties in a family which contain a rational point}.
\emph{J. Eur. Math. Soc.}, to appear.

\bibitem{LS16}
{D. Loughran, A. Smeets}, 
{Fibrations with few rational points}.
\emph{Geom. Func. Anal. (GAFA)}, {\bf26}(5) (2016), 1449--1482.


\bibitem{GMO08}
A. Gorodnik, F. Maucourant, H. Oh, Manin's conjecture on rational points of bounded height and adelic mixing.
\emph{Ann. Sci. \'{E}c. Norm. Sup\'{e}r.} {\bf41} (2008), 47--97. 

\bibitem{Kau00}
{I. Kausz},
A modular compactification of the general linear group.
\emph{Doc. Math.} {\bf 5}, (2000), 553--594.

\bibitem{MT11}
{G. Malle, D. Testerman},
\textit{Linear algebraic groups and finite groups of Lie type}.
Cambridge Studies in Advanced Mathematics, {\bf133}. Cambridge University Press, 2011.

\bibitem{Mil80}
{J. S. Milne},
\textit{\'{E}tale cohomology}.
Princeton Mathematical Series, {\bf33}. Princeton University Press, Princeton, N.J., 1980.

\bibitem{MW}
{C. Moeglin, J.-L. Waldspurger},
\textit{Spectral decomposition and Eisenstein series. Une paraphrase de l'\'{E}criture}.
Cambridge Tracts in Mathematics, {\bf113}. Cambridge University Press, 1995.

\bibitem{Muller}
{W. M\"uller},
{The trace class conjecture in the theory of automorphic Forms. II.}
\emph{Geom. Func. Anal. (GAFA)}, {\bf8}(2) (1998), 315--355.

\bibitem{NSW08}
{J. Neukirch, A. Schmidt, K. Wingberg}, 
\textit{Cohomology of Number Fields}, 
Second edition. Grundlehren der Mathematischen Wissenschafte {\bf323}. Springer-Verlag, Berlin, 2008.

\bibitem{Oh} 
{H. Oh},
{Uniform pointwise bounds for matrix coefficients of unitary representations
and applications to Kazhdan constants.}
\emph{Duke Math. J.} {\bf 113} (2002), 133--192.

\bibitem{Ono65}
{T. Ono}, 
{On the relative theory of Tamagawa numbers}.
\emph{Ann. of Math.} (2) {\bf 82} (1965), 88--111.

\bibitem{Pey95}
{E. Peyre}, 
{Hauteurs et mesures de Tamagawa sur les vari\'{e}ti\'{e}s de Fano}.
\emph{Duke Math. J.}, {\bf79}(1) (1995), 101--218.

\bibitem{PR94}
{V. Platonov, A. Rapinchuk}, \textit{Algebraic groups and number theory}.
Pure and Applied Mathematics, 139. Academic Press, Inc., Boston, MA, 1994.

\bibitem{Poo17}
B. Poonen, \textit{Rational points on varieties}.
Graduate Studies in Mathematics {\bf{186}}, 2017.

\bibitem{San81}
{J.-J. Sansuc}, 
{Groupe de Brauer et arithm\'{e}tique des groupes alg\'{e}briques lin\'{e}aires sur un corps de nombres}.
\emph{J. reine angew. Math.} {\bf327} (1981), 12--80.

\bibitem{Ser90}
{J.-P. Serre}, 
{Sp\'{e}cialisation des \'{e}l\'{e}ments de $\Br_2(\QQ(T_1,\ldots,T_n))$}.
\emph{C. R. Acad. Sci. Paris S\'{e}r. I Math.} {\bf311} (1990), no. 7, 397--402.

\bibitem{STBT_PGL2}
{J. Shalika, R. Takloo-Bighash, Y. Tschinkel},
{Rational points on compactifications of semi-simple groups of rank $1$}.
\emph{Progr. Math.} {Arithmetic of higher-dimensional algebraic varieties ({P}alo {A}lto, {CA}, 2002)} (2004), vol. 226, 205--233.

\bibitem{JAMS}
{J. Shalika, R. Takloo-Bighash, Y. Tschinkel},
{Rational points on compactifications of semi-simple groups}. 
\emph{J. Amer. Math. Soc.}  {\bf20} (2007), no. 4, 1135--1186. 

\bibitem{Sof16}
{E. Sofos}, 
{Serre's problem on the density of isotropic fibres in conic bundles}.
\emph{Proc. London Math. Soc.} {\bf113} (2016), 1--28.

\bibitem{Ten95}
{G. Tenenbaum}, 
\textit{Introduction to analytic and probabilistic number theory}. 
Cambridge University Press, 1995.

\bibitem{Warner}
{G. Warner},
\textit{Harmonic Analysis on Semi-Simple Lie Groups I}. 
Springer-Verlag, 1972.

\end{thebibliography}
\end{document}